\documentclass{article}

\input{preamble.tex}

\newcommand{\lm}[2]{\bar{#1}_{#2}}
\newcommand{\dop}{\mathrm{d}}
\renewcommand{\Re}{\mathbf{R}}
\renewcommand{\epsilon}{\varepsilon}
\newcommand{\vech}{\mathrm{vech}}

\newcommand{\E}[1]{\mathbf{E}\left[{#1}\right]}

\newcommand{\parens}[1]{\left({#1}\right)}

\newcommand{\crotchet}[1]{\left[{#1}\right]}

\newcommand{\tuborg}[1]{\left\{{#1}\right\}}

\newcommand{\norm}[1]{\left\|{#1}\right\|}
\newcommand{\abs}[1]{\left|{#1}\right|}

\newcommand{\cp}{\overset{P}{\to}}
\newcommand{\cl}{\overset{\mathcal{L}}{\to}}
\newcommand{\tr}{\mathrm{tr}}
\newcommand{\ip}[2]{#1\left\llbracket #2\right\rrbracket}

\newtheorem{theorem}{Theorem}

\theoremstyle{remark}
\newtheorem*{remark}{Remark}

\title{Adaptive test for ergodic diffusions plus noise}
\author[1]{Shogo H. Nakakita}
\author[1,2]{Masayuki Uchida}
\affil[1]{Graduate School of Engineering Science, Osaka University}
\affil[2]{Center for Mathematical Modeling and Data Science, Osaka University}
\date{\today}

\begin{document}
	
\maketitle

\begin{abstract}
We propose some parametric tests for ergodic diffusion-plus-noise model, which is a version of state-space modelling in statistics for stochastic diffusion equations. The test statistics are classified into three types: 
likelihood-ratio-type test statistic; Wald-type one; and Rao-type one. 
All the test statistics are constructed with quasi-likelihood-functions for local mean sequence of noised observation.
We also simulate the behaviour of them for several practical hypothesis tests and check the convergence in law of test statistics under null hypotheses and consistency of the test under alternative ones.
We apply the method for real data analysis of wind data, and examine some sets of the hypotheses mainly with respect to the structure of diffusion coefficient.
\end{abstract}

\section{Introduction}

Our research deals with the $d$-dimensional diffusion process being a solution for the following stochastic differential equation:
\begin{align*}
\dop X_t = b(X_t, \beta)\dop t + a(X_t, \alpha)\dop w_t,\ X_0=x_0,
\end{align*}
where $\tuborg{w_t}_{t\ge 0}$ is the $r$-dimensional Wiener process, $x_0$ is a $\Re^d$-valued random variable, $\alpha\in\Theta_1\subset\Re^{m_1}$, $\beta\in\Theta_2\subset\Re^{m_2}$, $\theta:=\parens{\alpha,\beta}$, $\Theta:=\Theta_1\times\Theta_2$ being the compact and convex parameter space, $a:\Re^d\times \Theta_1 \to \Re^d\otimes\Re^r$ and $b:\Re^d\times \Theta_2\to \Re^d$ are known functions. We assume that the true value of parameter $\theta^\star$ belongs to $\mathrm{Int}\parens{\Theta}$. 

We set the observational scheme same as \citep{NU17} such that for $i=0,\ldots,n$,
\begin{align*}
Y_{ih_n}=X_{ih_n}+\Lambda^{1/2}\epsilon_{ih_n},
\end{align*}
where $h_n$ is the discretised step satisfying $h_n\to 0$ and $T_n:=nh_n\to \infty$, $\Lambda\in\Re^d\otimes\Re^d$ is a positive semi-definite matrix, and $\tuborg{\epsilon_{ih_n}}_i$ is an i.i.d. sequence of random variables such that $\E{\epsilon_{ih_n}}=\mathbf{0}$, $\mathrm{Var}\parens{\epsilon_{ih_n}}=I_d$, and each component is independent of other components as well as $\tuborg{w_t}$ and $x_0$. Let us define $\Theta_{\epsilon}$ the compact and convex parameter space of $\vech \Lambda$, $\vartheta:=\parens{\theta,\theta_{\epsilon}}$, and $\Xi:=\Theta\times \Theta_{\epsilon}$.

The statistical framework for the analysis of time series data has been mainly based on discrete-time stochastic processes such as ARMA model (see \citep{BD91}). Those discrete-time models confront with some difficulties to express complex phenomena such that innovation term whose variance is dependent on state $X_{t}$ itself. One of the solutions for those difficulties is the modelling with stochastic differential equations, which flexibly describe the probabilistic perturbation dependent on $X_t$. The parametric inference for diffusion processes modelled with stochastic differential equations has been researched enthusiastically (e.g., see \citep{Fl89}, \citep{Y92}, \citep{BS95}, \citep{K95}, \citep{K97}, \citep{Y11}, \citep{UY12}, and \citep{UY14}). As is well known, parametric estimation for one model is not sufficient in the context of real data analysis; we need methodology to compare multiple parametric models in terms of goodness of fit (e.g., for i.i.d. case, see \citep{Fe96} and \citep{LR05}). \citep{KU14} proposes likelihood-ratio-type test statistic for discretely observed ergodic diffusions to examine parametric hypotheses such as \citep{Fe96} and shows the convergence in law of test statistics and the consistency of the test. As another approach, \citep{U10} researches contrast-based information criterion for ergodic diffusion processes with discretised observation scheme (see also \citep{FU14} and \citep{EM18}). These instruments to examine goodness-of-fit are important to see whether we are motivated to use the flexible modelling with stochastic differential equations.

The classical time series analysis itself also has instruments of complex modelling such as state-space model (see \citep{BD91}). One simple version of state-space modelling decomposes the randomness of observation into endogenous perturbation of the system of interest and exogenous noise which contaminates only observation and does not influence the system itself. The importance of this decomposition has attracted attention not only in the research of time series analysis but also that of statistics for stochastic differential equations. For instance, the existence of observation noise in high-frequency financial data called microstructure noise is one of the major research topics in financial econometrics. \citep{GlJ01a}, \citep{GlJ01b} and \citep{JLMPV09} researched the diffusion with noise contaminating observation in the observation framework such that $nh_n$ is fixed. The statistics for diffusion-plus-noise with the setting $nh_n\to\infty$ has been also researched, e.g., by \citep{Fa14}, \citep{Fa16}, and \citep{NU17}. \citep{Fa14} proposes the consistent estimator for the variance of noise and parameter of the diffusion process, and \citep{Fa16} construct the estimator for the parameter of the diffusion with asymptotic normality when the variance of noise is known. \citep{NU17} provides the estimator for both the parameter of the diffusion process and the variance of noise with asymptotic normality when the variance of noise is unknown. However, as discussed above, it is necessary to construct the way to compare the goodness-of-fit of candidate models in practice, and this research tries to achieve it with likelihood-ratio-type test statistics with quasi-likelihood-functions proposed by \citep{NU17} in the manner of \citep{KU14} discussing the same problem under the assumption that exogenous noise in observation does not exist.

We also analyse some real data with our methods besides theoretical construction of test statistics. Our data of interest is MetData (\citep{NWTC}) which represents wind velocity with high frequency observation. \citep{NU17} examines existence of noise in some partial data (the plot is shown in Figure \ref{MetData}) in MetData and shows statistical significance of the existence, which indicates the motivation to use diffusion-plus-noise modelling rather than diffusion modelling without observation noise. We also use the same data and see if the diffusion coefficient $a\parens{x,\alpha}$ is dependent on $x$ or not and check if we are motivated to utilise the flexible modelling of stochastic diffusion equations.

\begin{figure}[h]
	\begin{subfigure}{.5\textwidth}
		\centering
		\includegraphics[bb=0 0 720 480,width=7cm]{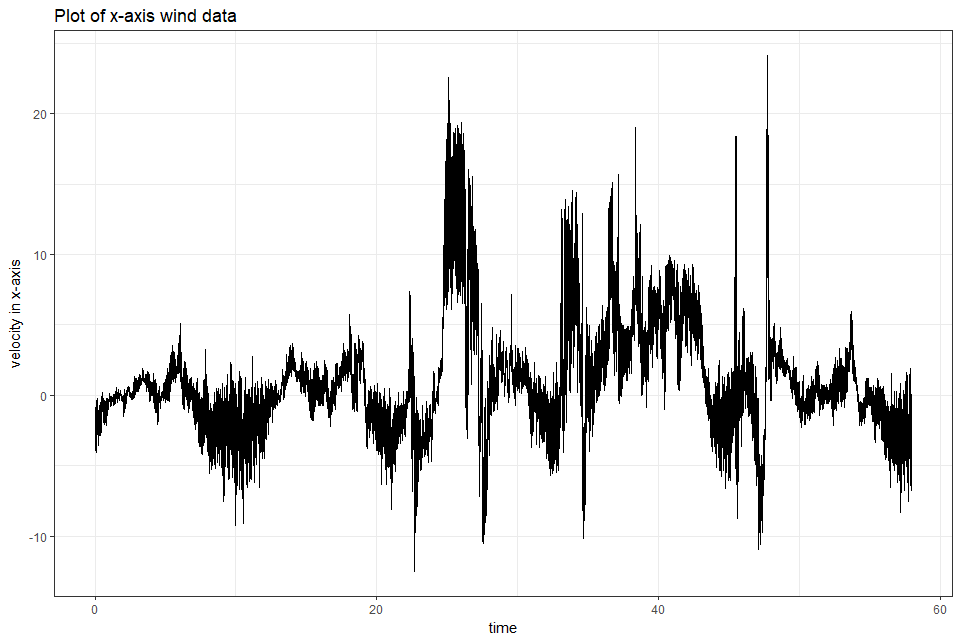}
	\end{subfigure}%
	\begin{subfigure}{.5\textwidth}
		\centering
		\includegraphics[bb=0 0 720 480,width=7cm]{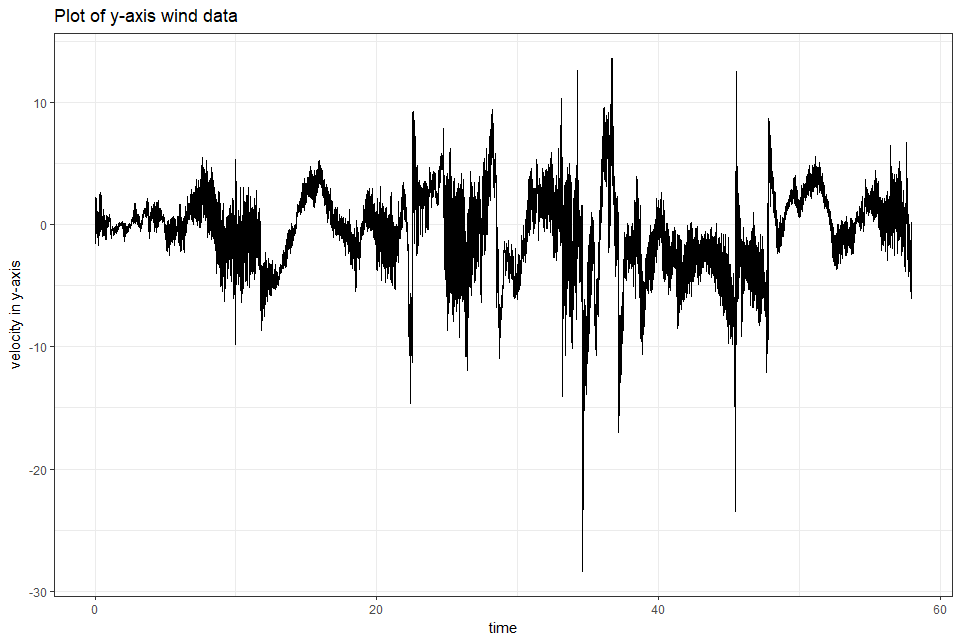}
	\end{subfigure}
	\caption{plot of wind velocity labelled Sonic x (left) and y (right) (119M) at the M5 tower from 00:00:00 on 1st July, 2017 to 20:00:00 on 5th July, 2017 with 0.05-second resolution \citep{NWTC}}\label{MetData}
\end{figure}

The paper composes of four parts: firstly, in Section 2, we show our assumption and notation before discussing concrete statements for parametric tests. In the second place, we state some theorems in Section 3 which show the asymptotic behaviour of adaptive likelihood-ratio-type test statistics, whose proofs are shown later. Section 4 examines the behaviour of statistics proposed in the previous section with computational simulation for 1-dimensional diffusion processes and 2-dimensional ones. Finally, in Section 5, MetData, a real data for wind with high-frequency observation, is used to see what our method concludes regarding the property of wind observed in data.
	
\section{Notation and assumption}

We set the following notations as \citep{NU17}.
\begin{enumerate}
	\item For a matrix $A$, $A^T$ denotes the transpose of $A$ and $A^{\otimes 2}:=AA^T$. For same size matrices $A$ and $B$, $\ip{A}{B}:=\tr\parens{AB^T}$. 
	\item For any vector $v$, $v^{(i)}$ denotes the $i$-th component of $v$. Similarly, $M^{(i,j)}$, $M^{(i,\cdot)}$ and $M^{(\cdot,j)}$ denote the $(i,j)$-th component, the $i$-th row vector and $j$-th column vector of a matrix $M$ respectively.
	\item $c(x,\alpha):=\parens{a(x,\alpha)}^{\otimes 2}.$
	\item $C$ is a positive generic constant independent of all other variables. If it depends on fixed other variables, e.g. an integer $k$, we will express as $C(k)$.
	\item $a(x):=a(x,\alpha^\star)$ and $b(x):=b(x,\beta^\star)$.
	\item Let us define $\vartheta:=\parens{\theta,\theta_\epsilon}\in \Xi$.
	\item A $\Re$-valued function $f$ on $\Re^d$ is a \textit{polynomial growth function} if there exists a constant $C>0$ such that for all $x\in \Re^d$,
	\begin{align*}
	\abs{f(x)}\le C\parens{1+\norm{x}}^C.
	\end{align*}
	$g:\Re^d\times \Theta\to \Re$ is a \textit{polynomial growth function uniformly in $\theta\in\Theta$} if there exists a constant $C>0$ such that for all $x\in \Re^d$,
	\begin{align*}
	\sup_{\theta\in\Theta}\abs{g(x,\theta)}\le C\parens{1+\norm{x}}^C.
	\end{align*}
	Similarly we say $h:\Re^d\times \Xi\to \Re$ is a \textit{polynomial growth function uniformly in $\vartheta\in\Xi$} if there exists a constant $C>0$ such that for all $x\in \Re^d$,
	\begin{align*}
	\sup_{\vartheta\in\Xi}\abs{h(x,\vartheta)}\le C\parens{1+\norm{x}}^C.
	\end{align*}
	\item For any $\Re$-valued sequence $u_n$, $R:\Theta\times\Re\times \Re^d\to \Re$ denotes a function with a constant $C$ such that
	\begin{align*}
	\abs{R(\theta,u_n,x)}\le Cu_n\parens{1+\norm{x}}^C
	\end{align*}
	for all $x\in\Re^d$ and $\theta\in\Theta$.
	\item Let us denote for any $\mu$-integrable function $f$ on $\Re^d$, $
	\mu(f(\cdot)) := \int f(x)\mu(\dop x).$
	\item 	We set \begin{align*}
	\mathbb{Y}_1(\alpha)&:=-\frac{1}{2}\nu_0\parens{\mathrm{tr}\parens{\parens{c(\cdot,\alpha)}^{-1}c(\cdot,\alpha^\star)-I_d}+\log\frac{\det c(\cdot,\alpha)}{\det c(\cdot,\alpha^\star)}},\\
	\mathbb{Y}_2(\beta)&:=
	-\frac{1}{2}\nu_0\parens{\ip{\parens{c(\cdot,\alpha^\star)}^{-1}}{\parens{b(\cdot,\beta)-b(\cdot,\beta^\star)}^{\otimes2}}}
	\end{align*}
	where $\nu_0$ is the invariant measure of $X$.
	\item $\cp$ and $\cl$ indicate convergence in probability and convergence in law respectively.
	\item For $f(x)$, $g(x,\theta)$ and $h(x,\vartheta)$, $f'(x):=\frac{\dop }{\dop x}f(x)$, $f''(x):=\frac{\dop^2}{\dop x^2}f(x)$, $\partial_{x}g(x,\theta):=\frac{\partial}{\partial x}g(x,\theta)$, $\partial_{\theta}g(x,\theta):=\frac{\partial}{\partial \theta}g(x,\theta)$, $\partial_{x}h(x,\vartheta):=\frac{\partial}{\partial x}h(x,\vartheta)$ and  $\partial_{\vartheta}h(x,\vartheta):=\frac{\partial}{\partial \vartheta}h(x,\vartheta)$.
	\item We define
	\begin{align*}
		J^{(2,2)}\parens{\alpha|\vartheta^{\star}}&:=
		\crotchet{\frac{1}{2}\nu_0\parens{\tr\tuborg{\parens{c}^{-1}\parens{\partial_{\alpha^{(i_1)}}c}\parens{c}^{-1}
					\parens{\partial_{\alpha^{(i_2)}}c}}(\cdot,\alpha)}
		}_{i_1,i_2},\\
		J^{(3,3)}(\beta|\vartheta^{\star})&:=
		\crotchet{\nu_0\parens{\ip{\parens{c}^{-1}}{
					\parens{\partial_{\beta^{(j_1)}}b}\parens{\partial_{\beta^{(j_2)}}b}^T}(\cdot,\alpha^\star,\beta)}}_{j_1,j_2},
	\end{align*}
	$J^{(2,2)}\parens{\vartheta^{\star}}:=J^{(2,2)}\parens{\alpha^{\star}|\vartheta^{\star}}$, 
	and $J^{(3,3)}(\vartheta^{\star}):=J^{(3,3)}(\beta^{\star}|\vartheta^{\star})$.
	\item $\lambda_{\min}(A)$ and $\lambda_{\max}(A)$ denote the minimum eigenvalue of a symmetric matrix $A$ and the maximum one respectively.
\end{enumerate}

We make the following assumptions.
\begin{enumerate}
	\item[(A1)] $b$ and $a$ are continuously differentiable for 4 times,  and the components of themselves as well as their derivatives are polynomial growth functions uniformly in $\theta\in\Theta$. Furthermore, there exists $C>0$ such that for all $x\in\Re^d$,
	\begin{align*}
	&\norm{b(x)}+\norm{b'(x)}+\norm{b''(x)}\le C(1+\norm{x}), \\ 
	&\norm{a(x)}+\norm{a'(x)}+\norm{a''(x)}\le C(1+\norm{x}). 
	\end{align*}
	\item[(A2)] $X$ is ergodic and the invariant measure $\nu_0$ has $k$-th moment for all $k>0$.
	\item[(A3)] For all $k>0$, $\sup_{t\ge0}\E{\norm{X_t}^k}<\infty$.
	\item[(A4)] For any $k>0$, $\epsilon_{ih_n}$ has $k$-th moment and the component of $\epsilon_{ih_n}$ are independent of the other components for all $i$, $\tuborg{w_t}$ and $x_0$. In addition, the marginal distribution of each component is symmetric.
	\item[(A5)] $\inf_{x,\alpha} \det c(x,\alpha)>0$.
	\item[(A6)] There exist positive constants $\chi$ and $\tilde{\chi}$ such that $\mathbb{Y}_1(\alpha)\le -\chi\norm{\alpha-\alpha^\star}^2$ and 
	$\mathbb{Y}_2(\beta)\le -\tilde{\chi}\norm{\beta-\beta^\star}^2$.
	\item[(A7)] The components of $b$, $a$, $\partial_xb$, $\partial_{\beta}b$, $\partial_xa$, $\partial_{\alpha}a$, $\partial_x^2b$, $\partial_{\beta}^2b$, $\partial_x\partial_{\beta}b$, $\partial_x^2a$, 
	$\partial_{\alpha}^2a$ and $\partial_x\partial_{\alpha}a$ are polynomial growth functions uniformly in $\theta\in\Theta$.
	\item[(AT)] $h_n=p_n^{-\tau},\ \tau\in(1,2)$ and $h_n\to0$, $p_n\to\infty$, $k_n\to\infty$, $\Delta_n=p_nh_n\to0$, $nh_n\to\infty$ as $n\to\infty$.
	\item[(R1)] It holds
	\begin{align*}
	\inf_{\alpha\in\Theta_1}\lambda_{\min}\parens{J^{(2,2)}\parens{\alpha|\vartheta^{\star}}}&>0,\\
	\inf_{\beta\in\Theta_1}\lambda_{\min}\parens{J^{(3,3)}\parens{\beta|\vartheta^{\star}}}&>0.
	\end{align*}
\end{enumerate}

\begin{remark}
	The assumption (AT) is a restriction of the assumption (AH) discussed in \citep{NU17} in terms of the space of the tuning parameter $\tau$; 
	(AH) sets it to be $(1,2]$.
	This assumption (AT) is necessary to match the asymptotic variance of the first derivatives of quasi-likelihood functions regarding parameters with the matrices
	to which the second derivatives of quasi-likelihoods with respect to parameter converge in probability.
\end{remark}

\section{Theorems and composition of parametric tests}

We consider the following statistical hypothesis testing problem, for $r\in\tuborg{1,\ldots,m_1+m_2}$,
\begin{align*}
H_0&\colon \theta^{(\lambda_1)}=\cdots=\theta^{(\lambda_{r_1})}=0,\\
H_1&\colon \text{not }H_0,
\end{align*}
where $\lambda_i\in\tuborg{1,\ldots,m_1+m_2}$ for all $i\in\tuborg{1,\ldots,r}$, and $\lambda_i<\lambda_j$ if $i\neq j$. 
Let us denote $r_1$ to be the number of elements $\lambda_i$ in $\tuborg{1,\ldots,m_1}$ and $r_2$ to be that of elements $\lambda_i$ in $\tuborg{m_1+1,\ldots,m_2}$. For simplicity, we also assume if $r_1>0$, then $\lambda_1=1,\ldots,\lambda_{r_1}=r_1$ and if $r_2>0$, then $\lambda_{m_1+1}=1,\ldots,\lambda_{m_1+r_2}=r_2$. That is, if $r_1>0$ and $H_0$ hold, then
\begin{align*}
	\alpha^{(1)}=\cdots=\alpha^{(r_1)}=0,
\end{align*}
and if $r_2>0$ and $H_0$ hold, then
\begin{align*}
	\beta^{(1)}=\cdots=\beta^{(r_2)}=0.
\end{align*}
We let $\Theta_{0,1}$ and $\Theta_{0,2}$ denote the parameter space of diffusion parameter and drift one under $H_0$. 
To compose the test statistic, let us define the following quasi-likelihood functions,
\begin{align*}
	&\mathbb{L}_{1,n}(\alpha|\Lambda):=-\frac{1}{2}\sum_{j=1}^{k_n-2}
	\parens{\ip{\parens{\frac{2}{3}\Delta_n c_n^{\tau}(\lm{Y}{j-1},\alpha,\Lambda)}^{-1}}{\parens{\lm{Y}{j+1}-\lm{Y}{j}}^{\otimes 2}}
		+\log\det \parens{ c_n^{\tau}(\lm{Y}{j-1},\alpha,\Lambda)}}, \\
	&\mathbb{L}_{2,n}(\beta|\Lambda,\alpha)
	:=-\frac{1}{2}\sum_{j=1}^{k_n-2}
	\ip{\parens{\Delta_nc_n^\tau(\lm{Y}{j-1},\alpha,\Lambda)}^{-1}}{\parens{\lm{Y}{j+1}-\lm{Y}{j}-\Delta_nb(\lm{Y}{j-1},\beta)}^{\otimes 2}},
\end{align*}
where $c_{n}^{\tau}\parens{x,\alpha,\Lambda}=c\parens{x,\alpha}+3\Delta_n^{\frac{2-\tau}{\tau-1}}\Lambda$.
Moreover, $\hat{\Lambda}_n$, $\hat{\alpha}_n$, $\tilde{\alpha}_n$, $\hat{\beta}_n$, and $\tilde{\beta}_n$ denote the estimators satisfying $\hat{\Lambda}_n=\frac{1}{2n}\sum_{i=0}^{n-1}\parens{Y_{(i+1)h_n}-Y_{ih_n}}^{\otimes2}$,
\begin{align*}
\mathbb{L}_{1,n}\parens{\hat{\alpha}_n|\hat{\Lambda}_n}&=\sup_{\alpha\in\Theta_{1}}\mathbb{L}_{1,n}\parens{\alpha|\hat{\Lambda}_n},\\
\mathbb{L}_{1,n}\parens{\tilde{\alpha}_n|\hat{\Lambda}_n}&=\sup_{\alpha\in\Theta_{0,1}}\mathbb{L}_{1,n}\parens{\alpha|\hat{\Lambda}_n},\\
\mathbb{L}_{2,n}\parens{\hat{\beta}_n|\hat{\Lambda}_n,\hat{\alpha}_n}&=\sup_{\beta\in\Theta_{2}}	\mathbb{L}_{2,n}\parens{\beta|\hat{\Lambda}_n,\hat{\alpha}_n},\\
\mathbb{L}_{2,n}\parens{\tilde{\beta}_n|\hat{\Lambda}_n,\hat{\alpha}_n}&=\sup_{\beta\in\Theta_{0,2}}	\mathbb{L}_{2,n}\parens{\beta|\hat{\Lambda}_n,\hat{\alpha}_n}.
\end{align*}

\begin{remark}
	$\lambda_i$ can take values only in $\tuborg{1,\ldots,m_1}$, or $\tuborg{m_1+1,\ldots,m_1+m_2}$. It indicates that we are able to test the diffusion parameter without drift one and vice versa.
\end{remark}

\subsection{Likelihood-ratio-type test}
We set the following likelihood-ratio-type test statistics:
\begin{align*}
\mathcal{T}_{1,n} &:= \frac{16}{9}\parens{\mathbb{L}_{1,n}\parens{\hat{\alpha}_n|\hat{\Lambda}_n}
	-\mathbb{L}_{1,n}\parens{\tilde{\alpha}_n|\hat{\Lambda}_n}} \\
&=-\frac{16}{9}\parens{\log \frac{\sup_{\alpha\in\Theta_{0,1}}\exp\parens{\mathbb{L}_{1,n}}(\alpha|\hat{\Lambda}_n)}
	{\sup_{\alpha\in\Theta_1}\exp\parens{\mathbb{L}_{1,n}}(\alpha|\hat{\Lambda}_n)}},\\
\mathcal{T}_{2,n} &:= 2\parens{\mathbb{L}_{2,n}\parens{\hat{\beta}_n|\hat{\Lambda}_n,\hat{\alpha}_n}
	-\mathbb{L}_{2,n}\parens{\tilde{\beta}_n|\hat{\Lambda}_n,\hat{\alpha}_n}} \\
&=-2\parens{\log \frac{\sup_{\beta\in\Theta_{0,2}}\exp\parens{\mathbb{L}_{2,n}}(\beta|\hat{\Lambda}_n,\hat{\alpha}_n)}
	{\sup_{\beta\in\Theta_2}\exp\parens{\mathbb{L}_{2,n}}(\beta|\hat{\Lambda}_n,\hat{\alpha}_n)}}.
\end{align*}
Note that if $r_i=0$, then $\mathcal{T}_{i,n}=0$ automatically for both of $i=1,2$.
\begin{theorem}\label{l0}
	Assume (A1)-(A7), (AT), $H_0$ and $k_n\Delta_n^2\to 0$ hold. Then we have
	\begin{align*}
	\mathcal{T}_{1,n} + \mathcal{T}_{2,n} &\cl \chi_{r}^2.
	\end{align*}
\end{theorem}
The consistency of the likelihood-ratio-type test holds because of the following theorem.
\begin{theorem}\label{l1}
	Assume (A1)-(A7), (AT) and $H_1$ hold. Then for all $M>0$,
	\begin{align*}
		P\parens{\mathcal{T}_{1,n} + \mathcal{T}_{2,n}\le M}\to 0.
	\end{align*}
\end{theorem}

\subsection{Other types of parametric test}

In addition, we also consider Rao-type test statistics:
\begin{align*}
\mathcal{R}_{1,n}&:=\parens{\frac{1}{\sqrt{k_n}}\partial_{\alpha}\mathbb{L}_{1,n}\parens{\tilde{\alpha}_n|\hat{\Lambda}_n}}
\parens{-\frac{9}{8k_n}\partial_{\alpha}^2\mathbb{L}_{1,n}\parens{\hat{\alpha}_n|\hat{\Lambda}_n}}^{-1}\parens{\frac{1}{\sqrt{k_n}}\partial_{\alpha}\mathbb{L}_{1,n}\parens{\tilde{\alpha}_n|\hat{\Lambda}_n}}^T,\\
\mathcal{R}_{2,n}&:=\parens{\frac{1}{\sqrt{T_n}}\partial_{\beta}\mathbb{L}_{2,n}\parens{\tilde{\beta}|\hat{\Lambda}_n,\hat{\alpha}_n}}
\parens{-\frac{1}{T_n}\partial_{\beta}^2\mathbb{L}_{2,n}\parens{\hat{\beta}_n|\hat{\Lambda}_n,\hat{\alpha}_n}}^{-1}
\parens{\frac{1}{\sqrt{T_n}}\partial_{\beta}\mathbb{L}_{2,n}\parens{\tilde{\beta}|\hat{\Lambda}_n,\hat{\alpha}_n}}^T.
\end{align*}

\begin{theorem}\label{r0}
	Assume (A1)-(A7), (AT), $H_0$ and $k_n\Delta_n^2\to 0$ hold. Then we have
	\begin{align*}
	\mathcal{R}_{1,n} + \mathcal{R}_{2,n} &\cl \chi_{r}^2.
	\end{align*}
\end{theorem}

\begin{theorem}\label{r1}
	Assume (A1)-(A7), (AT), (R1) and $H_1$ hold. Then for all $M>0$,
	\begin{align*}
	P\parens{\mathcal{R}_{1,n} + \mathcal{R}_{2,n}\le M}\to 0.
	\end{align*}
\end{theorem}

Furthermore, let us define the following Wald-type statistics:
\begin{align*}
	\mathcal{W}_{1,n}&=k_n\parens{\hat{\alpha}_n-\tilde{\alpha}_n}^T
	\parens{-\frac{9}{8k_n}\partial_{\alpha}^2\mathbb{L}_{1,n}\parens{\hat{\alpha}_n|\hat{\Lambda}_n}}
	\parens{\hat{\alpha}_n-\tilde{\alpha}_n}\\
	\mathcal{W}_{2,n}&=T_n\parens{\hat{\beta}_n-\tilde{\beta}_n}^T
	\parens{-\frac{1}{T_n}\partial_{\beta}^2\mathbb{L}_{2,n}\parens{\hat{\beta}_n|\hat{\Lambda}_n,\hat{\alpha}_n}}
	\parens{\hat{\beta}_n-\tilde{\beta}_n}.
\end{align*}

\begin{theorem}\label{w0}
	Assume (A1)-(A7), (AT), $H_0$ and $k_n\Delta_n^2\to 0$ hold. Then we have
	\begin{align*}
	\mathcal{W}_{1,n} + \mathcal{W}_{2,n} &\cl \chi_{r}^2.
	\end{align*}
\end{theorem}

\begin{theorem}\label{w1}
	Assume (A1)-(A7), (AT) and $H_1$ hold. Then for all $M>0$,
	\begin{align*}
	P\parens{\mathcal{W}_{1,n} + \mathcal{W}_{2,n}\le M}\to 0.
	\end{align*}
\end{theorem}

\section{Simulation study}

\subsection{1-dimensional diffusion}
We consider the diffusion process with the following SDE:
\begin{align*}
	\dop X_t = \parens{\beta^{(1)}X_t+\beta^{(2)}}\dop t + \parens{\alpha^{(1)}+\frac{\alpha^{(2)}\parens{X_t}^2}{1+\parens{X_t}^2}}\dop w_t,\ X_0 = 0
\end{align*}
and the simulation setting throughout this subsection is shown in the next table;
\begin{table}[h]
	\centering
	\begin{tabular}{c|ccccccc}
		parameter & $n$ & $h_n$ & $T_n$ & $\tau$ & $p_n$ & $k_n$ & iteration\\\hline
		value & $10^6$ & $6.31\times10^{-5}$ & $63.1$ & $1.9$ & $162$ & $6172$ & $10000$
	\end{tabular}
	\caption{Simulation setting in section 4.1}
\end{table}
and that for noise is fixed as $\epsilon_{ih_n}\overset{\mathrm{i.i.d.}}{\sim}N\parens{0, 1}$, and $\Lambda_{\star}=10^{-3}$ whose size is so large that test for noise detection in \citep{NU17} detect with high probability.

\subsubsection{Test for diffusion parameters}
We propose the following test:
\begin{align*}
H_0:&\text{ }\alpha^{(2)}=0,\\
H_1:&\text{ not }H_0.
\end{align*}
When $H_0$ holds, we can interpret that an Ornstein-Uhlenbeck processes describe data enough and our full model is no more useful. 
To the contrary, rejection of $H_0$ indicates that OU processes are not enough to fit the data and our full model with diffusion coefficient dependent on state, which is difficult for traditional time series model to express, is more appropriate to express the observation.

We do the two simulations with different true value of parameters: one is with true value
\begin{align*}
	\alpha^{\star}=\crotchet{1,0},\ \beta^{\star}=\crotchet{-1,1},
\end{align*}
where $H_0$ holds; and the other is with
\begin{align*}
	\alpha^{\star}=\crotchet{1,1},\ \beta^{\star}=\crotchet{-1,1},
\end{align*}
where $H_1$ is true. Then the test statistic $\mathcal{T}_{1,n}$ behaves as shown in table \ref{table_411}. Note that $\chi_{r}^2(p)$ indicates $p$ is the upper $p$-point of $\chi^2$ distribution with degree of freedom $r$.
The figure \ref{edf_411_H0} depicts the empirical distribution function of test statistics and the theoretical one of $\chi_{1}^2$.
We can see that our likelihood-ratio-type statistic has asymptotic distribution as we have shown from these results, at least with respect to diffusion parameters. Hence we can conclude that our test has consistency in this simulation.

\begin{table}[h]
	\centering
	\begin{tabular}{c|cccc}
		& \multicolumn{4}{c}{empirical ratio of $\mathcal{T}_{1,n}$ larger than...}\\
		& $\chi_{1}^2(0.10)$ & $\chi_{1}^2(0.05)$ & $\chi_{1}^2(0.01)$ & $\chi_{1}^2(0.001)$ \\\hline
		$H_0$ is true: & 0.0987 & 0.0516 & 0.0099 & 0.0015\\
		$H_1$ is true: & $1$ & $1$ & $1$ & $1$
	\end{tabular}
	\caption{Simulation result under $H_0$ and $H_1$ in section 4.1.1}\label{table_411}
\end{table}

\subsubsection{Test for drift parameters}
We also consider the parametric test for drift parameters: let us consider the next hypotheses and the statistical test:
\begin{align*}
H_0:&\text{ }\beta^{(2)}=0,\\
H_1:&\text{ not }H_0.
\end{align*}
and again we see the behaviour of likelihood-ratio-type statistic $\mathcal{T}_{2,n}$. 
As seen, the setting questions whether our model is symmetric with respect to $0$ and it can be of interest 
when state with $0$ value is interpreted as the neutral one (e.g., wind velocity).
In the first place, we experiment our statistic with the true value
\begin{align*}
\alpha^{\star}=\crotchet{1,1},\ \beta^{\star}=\crotchet{-1,0},
\end{align*}
where $H_0$ is true; in the second place, we do with
\begin{align*}
\alpha^{\star}=\crotchet{1,1},\ \beta^{\star}=\crotchet{-1,1},
\end{align*}
where $H_1$ holds. 
The simulation result is summarised in the table \ref{table_412}. The empirical distribution function of $\mathcal{T}_{2,n}$ is shown in Figure \ref{edf_412_H0} combined with the theoretical distribution of $\chi_1^2$.
These results show that our test statistic $\mathcal{T}_{2,n}$ has asymptotic distribution as we have shown theoretically.
\begin{table}[h]
	\centering
	\begin{tabular}{c|cccc}
		& \multicolumn{4}{c}{empirical ratio of $\mathcal{T}_{2,n}$ larger than...}\\
		& $\chi_{1}^2(0.10)$ & $\chi_{1}^2(0.05)$ & $\chi_{1}^2(0.01)$ & $\chi_{1}^2(0.001)$ \\\hline
		$H_0$ is true: & 0.1086 & 0.0545 & 0.0113 & 0.001\\
		$H_1$ is true: & 1 &  1 & 0.9998 & 0.9881
	\end{tabular}
	\caption{Simulation result under $H_0$ and $H_1$ in section 4.1.2}\label{table_412}
\end{table}

\begin{figure}[p]
	\centering
	\includegraphics[width=\linewidth, bb = 0 0 720 480]{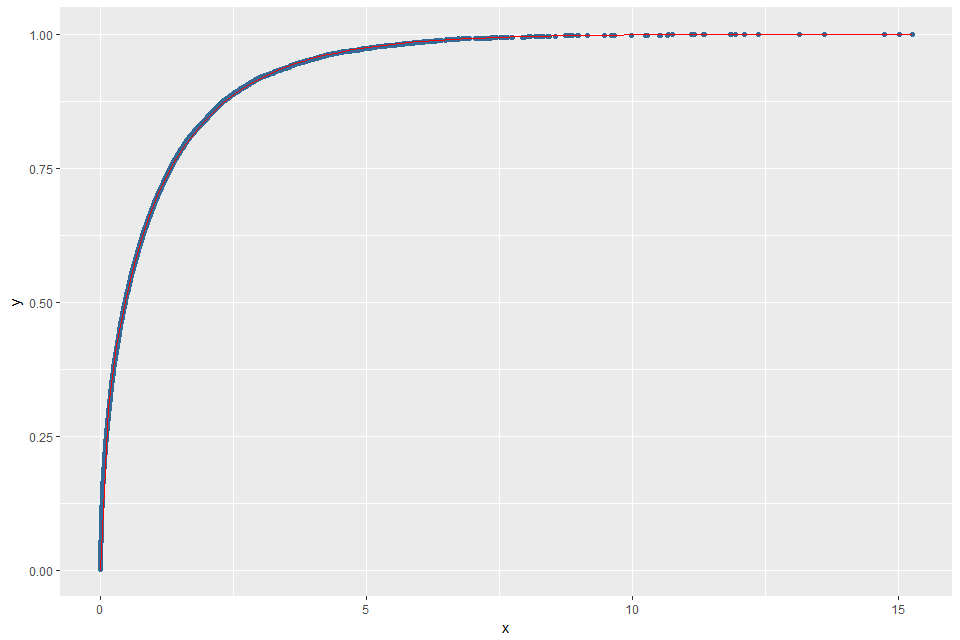}
	\caption{empirical distribution function of $\mathcal{T}_{1,n}$ (blue point) under $H_0$ and distribution function of $\chi_1^2$ (red line), section 4.1.1}\label{edf_411_H0}
	
	\includegraphics[width=\linewidth, bb = 0 0 720 480]{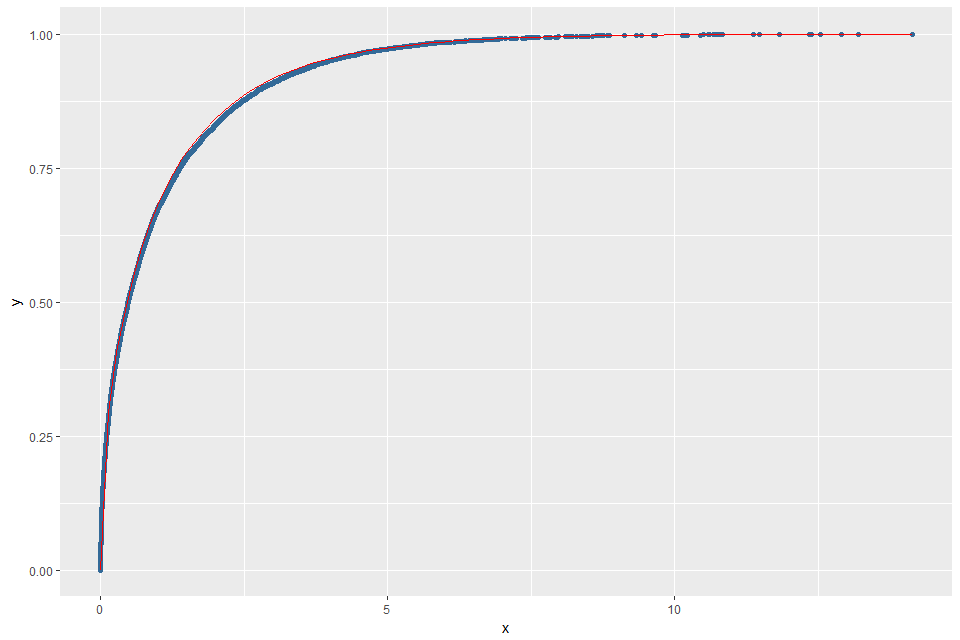}
	\caption{empirical distribution function of $\mathcal{T}_{2,n}$ (blue point) under $H_0$ and distribution function of $\chi_1^2$ (red line), section 4.1.2}\label{edf_412_H0}
\end{figure}

\subsection{2-dimensional diffusion}

We also deal with a multidimensional example of diffusion process such that
\begin{align*}
\begin{cases}
\dop\crotchet{\begin{matrix}
	X_t^{(1)}\\
	X_t^{(2)}
	\end{matrix}} = \parens{\crotchet{\begin{matrix}
		\beta^{(1)} & \beta^{(2)}\\
		\beta^{(4)} & \beta^{(5)}
		\end{matrix}}\crotchet{\begin{matrix}
		X_t^{(1)}\\
		X_t^{(2)}
		\end{matrix}}+\crotchet{\begin{matrix}
		\beta^{(3)}\\
		\beta^{(6)}
		\end{matrix}}}\dop t \\
\qquad\qquad\qquad+ \crotchet{\begin{matrix}
	\alpha^{(1)}+\frac{\alpha^{(2)}\parens{X_t^{(1)}}^2}{1+\parens{X_t^{(1)}}^2} + \frac{\alpha^{(3)}\parens{X_t^{(2)}}^2}{1+\parens{X_t^{(2)}}^2} 
	& \parens{\alpha^{(1)}\alpha^{(4)}}^{1/2}\alpha^{(7)}\\
	\parens{\alpha^{(1)}\alpha^{(4)}}^{1/2}\alpha^{(7)} & \alpha^{(4)}+\frac{\alpha^{(5)}\parens{X_t^{(1)}}^2}{1+\parens{X_t^{(1)}}^2} + \frac{\alpha^{(6)}\parens{X_t^{(2)}}^2}{1+\parens{X_t^{(2)}}^2}
	\end{matrix}}\dop w_t,
\\
\crotchet{\begin{matrix}
	X_0^{(1)}\\
	X_0^{(2)}
	\end{matrix}} = \crotchet{\begin{matrix}
	0\\
	0
	\end{matrix}}.
\end{cases}
\end{align*}

\begin{table}[h]
	\centering
	\begin{tabular}{c|ccccccc}
		parameter & $n$ & $h_n$ & $T_n$ & $\tau$ & $p_n$ & $k_n$ & iteration\\\hline
		value & $10^6$ & $6.31\times10^{-5}$ & $63.1$ & $1.9$ & $162$ & $6172$ & $2000$
	\end{tabular}
	\caption{Simulation setting in section 4.2}
\end{table}
With respect to the noise, we set $\epsilon_{ih_n}\overset{i.i.d.}{\sim}N\parens{0,I_2}$ and $\Lambda_{\star}=10^{-3}I_2$.

\subsubsection{Ornstein-Uhlenbeck test}
The hypotheses of interest in this section are
\begin{align*}
	H_0:&\alpha^{(2)}=\alpha^{(3)}=\alpha^{(5)}=\alpha^{(6)}=0,\\
	H_1:&\text{ not }H_0.
\end{align*}
This set of the hypotheses is for seeing whether the latent process $X$ is an Ornstein-Uhlenbeck process or not. If $H_0$ is rejected, it indicates that the 'innovation' in traditional time series analysis is dependent on the state $X$ and it is the situation where statistics for diffusion process prepares stronger tools for analysis. Hence the test with these hypothesis can test to use high-frequency observation framework and diffusion modelling for the target latent process.

We implement the two sorts of simulation: the first one is with the true parameter
\begin{align*}
	\alpha^{\star}=\crotchet{4, 0, 0, 4, 0, 0, -0.2},\ \beta^{\star}=\crotchet{-1, -0.1, 0, -0.1, -1, 0}
\end{align*}
where $H_0$ holds; and the second one is with
\begin{align*}
	\alpha^{\star}=\crotchet{4, 1, 1, 4, 1, 1, -0.2},\ \beta^{\star}=\crotchet{-1, -0.1, 1, -0.1, -1, 1},
\end{align*}
where $H_1$ is true. The summary of empirical ratio of $\mathcal{T}_{1,n}$ exceeding some critical values is shown in table \ref{table_421} 
and the plot of empirical distribution is drawn in figure \ref{edf_421_H0} with theoretical one. 

\begin{table}[h]
	\centering
	\begin{tabular}{c|cccc}
		& \multicolumn{4}{c}{empirical ratio of $\mathcal{T}_{2,n}$ larger than...}\\
		& $\chi_{4}^2(0.10)$ & $\chi_{4}^2(0.05)$ & $\chi_{4}^2(0.01)$ & $\chi_{4}^2(0.001)$ \\\hline
		$H_0$ is true: & 0.1085 & 0.0525 & 0.01&  0.001\\
		$H_1$ is true: & 1 & 1 & 1 & 1
	\end{tabular}
	\caption{Simulation result under $H_0$ and $H_1$ in section 4.2.1}\label{table_421}
\end{table}

\begin{figure}[h]
	\centering
	\includegraphics[width=\linewidth, bb = 0 0 720 480]{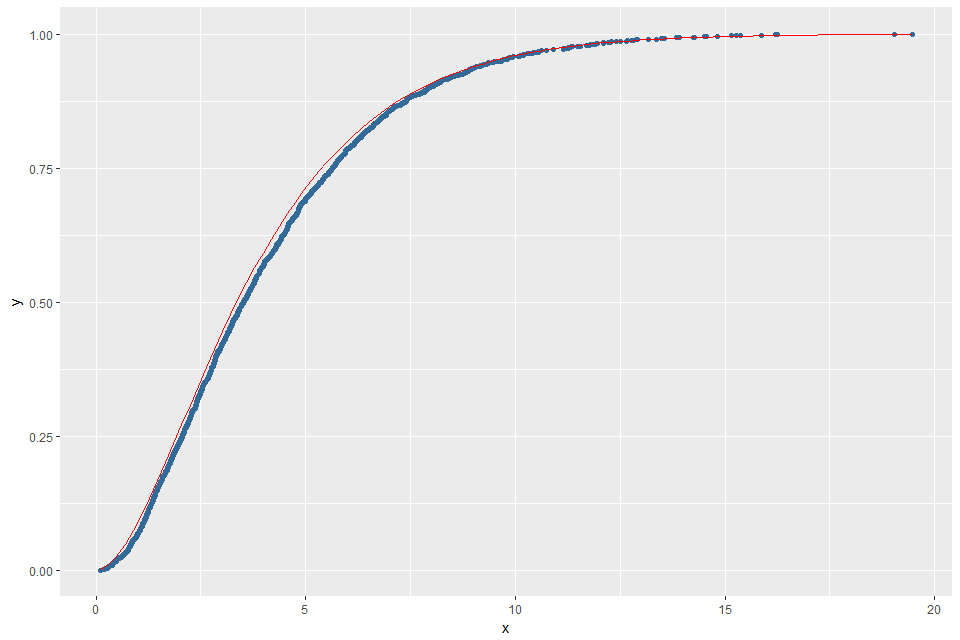}
	\caption{empirical distribution function of $\mathcal{T}_{2,n}$ (blue point) under $H_0$ and distribution function of $\chi_2^2$ (red line), section 4.2.1}\label{edf_421_H0}
\end{figure}

\subsubsection{Centricity test}

We consider the following hypotheses:
\begin{align*}
H_0:&\text{ }\beta^{(3)}=\beta^{(6)}=0,\\
H_1:&\text{ not }H_0.
\end{align*} This set of the hypotheses is a multivariate version of that discussed in section 4.1.2.

Firstly We set the true parameter
\begin{align*}
\alpha^{\star}=\crotchet{1, 1, 1, 1, 1, 1, 0.1},\ \beta^{\star}=\crotchet{-1, -0.1, 0, -0.1, -1, 0}
\end{align*}
where $H_0$ holds.
In the second place, we consider the simulation with the true value
\begin{align*}
\alpha^{\star}=\crotchet{1, 1, 1, 1, 1, 1, 0.1},\ \beta^{\star}=\crotchet{-1, -0.1, 1, -0.1, -1, 1},
\end{align*}
where $H_1$ is true. The result is summarised in table \ref{table_422} 
and the empirical distribution is plotted in figure \ref{edf_422_H0} with theoretical one. 

\begin{table}[h]
	\centering
	\begin{tabular}{c|cccc}
		& \multicolumn{4}{c}{empirical ratio of $\mathcal{T}_{2,n}$ larger than...}\\
		& $\chi_{2}^2(0.10)$ & $\chi_{2}^2(0.05)$ & $\chi_{2}^2(0.01)$ & $\chi_{2}^2(0.001)$ \\\hline
		$H_0$ is true: & 0.1105 &  0.0505 & 0.015 &  0.002\\
		$H_1$ is true: & 1 & 0.9995 & 0.999 & 0.979
	\end{tabular}
	\caption{Simulation result under $H_0$ and $H_1$ in section 4.2.2}\label{table_422}
\end{table}

\begin{figure}[p]
	\centering
	\includegraphics[width=\linewidth, bb = 0 0 720 480]{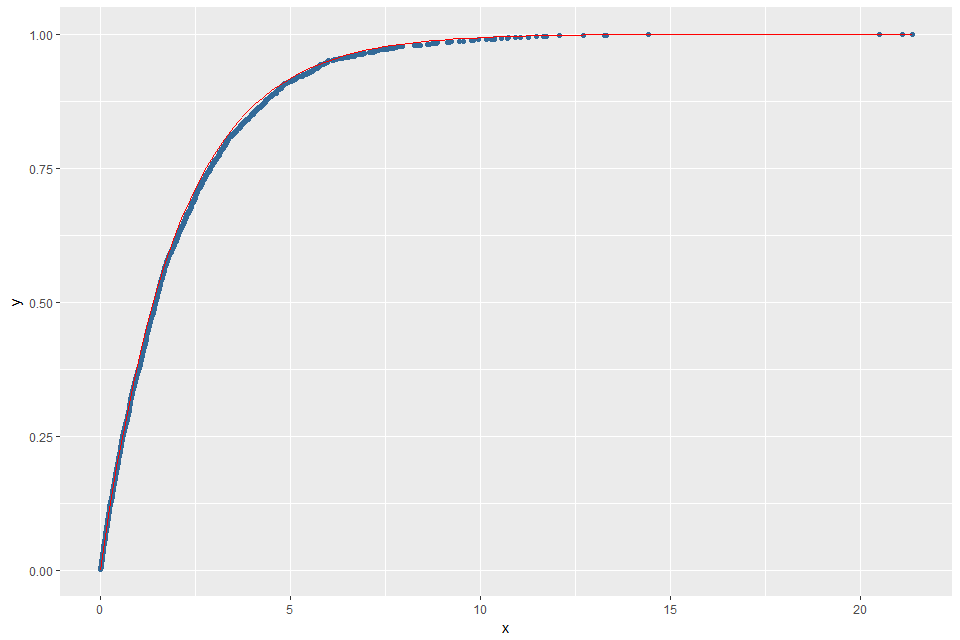}
	\caption{empirical distribution function of $\mathcal{T}_{2,n}$ (blue point) under $H_0$ and distribution function of $\chi_2^2$ (red line), section 4.2.2}\label{edf_422_H0}
	\includegraphics[width=\linewidth, bb = 0 0 720 480]{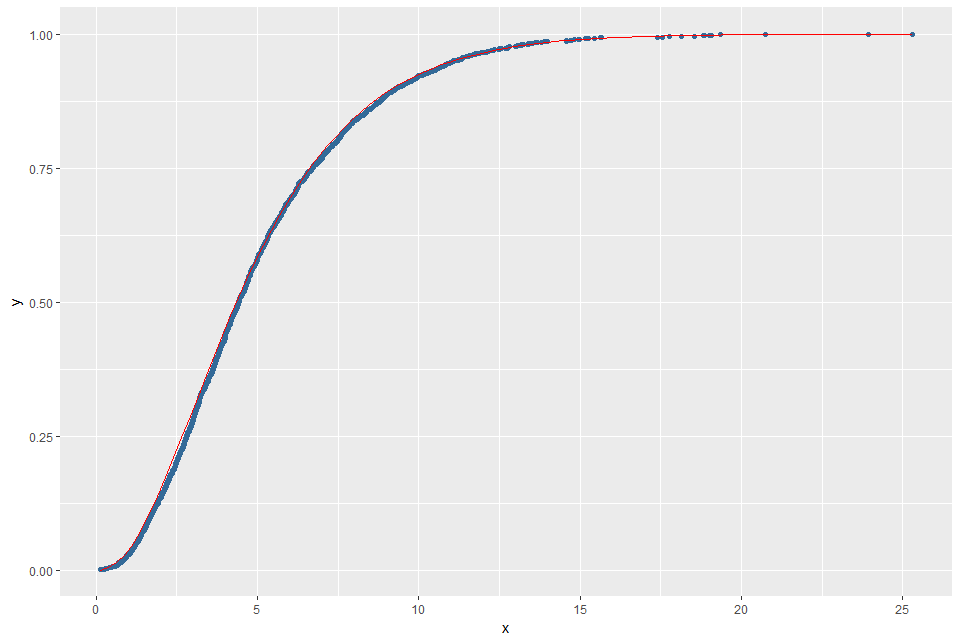}
	\caption{empirical distribution function of $\mathcal{T}_{1,n}+\mathcal{T}_{2,n}$ (blue point) under $H_0$ and distribution function of $\chi_5^2$ (red line), section 4.2.3}\label{edf_423_H0}
\end{figure}

\subsubsection{Independence test}

We consider the following hypotheses:
\begin{align*}
H_0:&\text{ }\alpha^{(3)}=\alpha^{(5)}=\alpha^{(7)}=\beta^{(2)}=\beta^{(4)}=0,\\
H_1:&\text{ not }H_0.
\end{align*}

Firstly we set the true parameter
\begin{align*}
\alpha^{\star}=\crotchet{4, 1, 0, 4, 0, 1, 0},\ \beta^{\star}=\crotchet{-1, 0, 1, 0, -1, 1},
\end{align*}
where $H_0$ holds.
In the second place, we consider the simulation with the true value
\begin{align*}
\alpha^{\star}=\crotchet{4, 1, 1, 4, 1, 1, -0.2},\ \beta^{\star}=\crotchet{-1, -0.1, 1, -0.1, -1, 1},
\end{align*}
where $H_1$ is true. The results are shown in table \ref{table_423} and figure \ref{edf_423_H0}.

\begin{table}[h]
	\centering
	\begin{tabular}{c|cccc}
		& \multicolumn{4}{c}{empirical ratio of $\mathcal{T}_{1,n}+\mathcal{T}_{2,n}$ larger than...}\\
		& $\chi_{5}^2(0.10)$  & $\chi_{5}^2(0.05)$ & $\chi_{5}^2(0.01)$ & $\chi_{5}^2(0.001)$ \\\hline
		$H_0$ is true: & 0.1045 & 0.0505 & 0.0095 & 0.0015\\
		$H_1$ is true: & $1$ & $1$ & $1$ & $1$ 
	\end{tabular}
	\caption{Simulation result under $H_0$ and $H_1$ in section 4.2.3}\label{table_423}
\end{table}

\section{Real data analysis}
As \citep{NU17}, we do the statistical analysis for wind data named MetData provided by National Wind Technology Center in US \citep{NWTC}. 
We focus on the 2-dimensional data with 0.05-second resolution representing wind velocity labelled Sonic x and Sonic y (119M) at the M5 tower, and analyse two dataset with different observation term: the first one is from 00:00:00 on 1st July, 2017 to 20:00:00 on 5th July, 2017 as \citep{NU17}; 
and the second one is from 00:00:00 on 1st April, 2016 to 00:00:00 on 21th April, 2016. 
With respect to the time unit, we set 2 hours for both datasets.

Our full model for both data is as follows:
\begin{align*}
	\begin{cases}
		\dop\crotchet{\begin{matrix}
				X_t^{(1)}\\
				X_t^{(2)}
		\end{matrix}} = \parens{\crotchet{\begin{matrix}
					\beta^{(1)} & \beta^{(2)}\\
					\beta^{(4)} & \beta^{(5)}
			\end{matrix}}\crotchet{\begin{matrix}
					X_t^{(1)}\\
					X_t^{(2)}
			\end{matrix}}+\crotchet{\begin{matrix}
					\beta^{(3)}\\
					\beta^{(6)}
		\end{matrix}}}\dop t \\
		\qquad\qquad\qquad+ \crotchet{\begin{matrix}
				\alpha^{(1)}+\frac{\alpha^{(2)}\parens{X_t^{(1)}}^2}{1+\parens{X_t^{(1)}}^2} + \frac{\alpha^{(3)}\parens{X_t^{(2)}}^2}{1+\parens{X_t^{(2)}}^2} 
				& \parens{\alpha^{(1)}\alpha^{(4)}}^{1/2}\alpha^{(7)}\\
				\parens{\alpha^{(1)}\alpha^{(4)}}^{1/2}\alpha^{(7)} & \alpha^{(4)}+\frac{\alpha^{(5)}\parens{X_t^{(1)}}^2}{1+\parens{X_t^{(1)}}^2} + \frac{\alpha^{(6)}\parens{X_t^{(2)}}^2}{1+\parens{X_t^{(2)}}^2}
		\end{matrix}}\dop w_t,
		\\
		\crotchet{\begin{matrix}
				X_0^{(1)}\\
				X_0^{(2)}
		\end{matrix}} = \crotchet{\begin{matrix}
		x_0^{(1)}\\
		x_0^{(2)}
		\end{matrix}}
	\end{cases}
\end{align*}

The some settings such that $n$, $h_n$ and $\tau$ are shown in the following table.
\begin{table}[h]
	\centering
	\begin{tabular}{c|cccccc}
	 & $n$ & $h_n$ & $T_n$ & $\tau$ & $p_n$ & $k_n$\\\hline
		July, 2017 & $8352000$ & $6.94\times10^{-6}$ & $58$ & $1.9$ & $518$ & $16123$ \\
		April, 2016 & $34560000$ & $6.94\times10^{-6}$ & $240$ & $1.9$ & $518$ & $66718$
	\end{tabular}
	\caption{Simulation setting in section 4.2}
\end{table}

\subsection{Data analysis for MetData in July, 2017}
The fitting of this full model with the local mean method is
\begin{align*}
	\begin{cases}
	\dop\crotchet{\begin{matrix}
		X_t^{(1)}\\
		X_t^{(2)}
		\end{matrix}} = \parens{\crotchet{\begin{matrix}
			-2.59 & -0.758\\
			-0.280 & -3.12
			\end{matrix}}\crotchet{\begin{matrix}
			X_t^{(1)}\\
			X_t^{(2)}
			\end{matrix}}+\crotchet{\begin{matrix}
			-0.625\\
			-0.763
			\end{matrix}}}\dop t \\
	\qquad\qquad+ \crotchet{\begin{matrix}
		3.13+9.07\frac{\parens{X_t^{(1)}}^2}{1+\parens{X_t^{(1)}}^2} + 4.10\frac{\parens{X_t^{(2)}}^2}{1+\parens{X_t^{(2)}}^2} 
		& \parens{3.13}^{1/2}\parens{3.17}^{1/2}\parens{-0.0763}\\
		\parens{3.13}^{1/2}\parens{3.17}^{1/2}\parens{-0.0763} & 3.17+8.54\frac{\parens{X_t^{(1)}}^2}{1+\parens{X_t^{(1)}}^2} + 3.59\frac{\parens{X_t^{(2)}}^2}{1+\parens{X_t^{(2)}}^2}
		\end{matrix}}\dop w_t,
	\\
	\crotchet{\begin{matrix}
		X_0^{(1)}\\
		X_0^{(2)}
		\end{matrix}} = \crotchet{\begin{matrix}
		x_0^{(1)}\\
		x_0^{(2)}
		\end{matrix}}
	\end{cases}
\end{align*}
We already have the result that the dataset is contaminated by noise with significance level $\alpha\ge10^{-16}$ (see \citep{NU17}). Hence it is reasonable to adopt the parameter estimation using local mean methods. 

Firstly, we do the statistical test for the set of the hypotheses
\begin{align*}
	H_0:&\alpha^{(2)}=\alpha^{(3)}=\alpha^{(5)}=\alpha^{(6)}=0,\\
	H_1:&\text{ not }H_0.
\end{align*}
which examines that the wind velocity can be expressed by Ornstein-Uhlenbeck process sufficiently. In the second place, we examine
\begin{align*}
	H_0:&\text{ }\beta^{(3)}=\beta^{(6)}=0,\\
	H_1:&\text{ not }H_0.
\end{align*}
the test for centrality discussed in Section 4.1.2 and 4.2.1. Finally, we consider the test for independence, that is,
\begin{align*}
	H_0:&\text{ }\alpha^{(3)}=\alpha^{(5)}=\alpha^{(7)}=\beta^{(2)}=\beta^{(4)}=0,\\
	H_1:&\text{ not }H_0.
\end{align*}
which is the topic in Section 4.2.2.

The results of the likelihood-ratio-type tests are shown in table \ref{table5}. We can conclude that the wind velocity cannot be fitted by OU process enough compared to our full model with common significance level.
With respect to the centricity, we cannot reject $H_0$ even with the significance level $0.10$ and hence there is no validity to regard the wind velocity is symmetric with respect to the zero vector.
What is more, the result of the test for independence indicates that it is meaningful to model velocity jointly with commonly used significance level.

\begin{table}[h]
	\centering
	\begin{tabular}{c|r|l}
		test & \multicolumn{1}{c}{test statistic} & \multicolumn{1}{|c}{$p$-value}\\\hline
		OU & $\mathcal{T}_{1,n}=6603.819$ & $p<10^{-16}$\\
		centricity & $\mathcal{T}_{2,n}=0.745618$ & $p=0.3112$\\
		independence & $\mathcal{T}_{1,n}+\mathcal{T}_{2,n}=3395.082$ & $p<10^{-16}$
	\end{tabular}
	\caption{Summary of the tests in Section 5}\label{table5}
\end{table}

\subsection{Data analysis for MetData in April, 2016}
The fitting for the second dataset results in
\begin{align*}
	\begin{cases}
	\dop\crotchet{\begin{matrix}
		X_t^{(1)}\\
		X_t^{(2)}
		\end{matrix}} = \parens{\crotchet{\begin{matrix}
			-2.40 & -0.657\\
			-0.677 & -3.84
			\end{matrix}}\crotchet{\begin{matrix}
			X_t^{(1)}\\
			X_t^{(2)}
			\end{matrix}}+\crotchet{\begin{matrix}
			4.57\\
			1.97
			\end{matrix}}}\dop t \\
	\qquad\qquad+ \crotchet{\begin{matrix}
		6.18+12.28\frac{\parens{X_t^{(1)}}^2}{1+\parens{X_t^{(1)}}^2} + 0.78\frac{\parens{X_t^{(2)}}^2}{1+\parens{X_t^{(2)}}^2} 
		& \parens{6.18}^{1/2}\parens{5.81}^{1/2}\parens{-0.0665}\\
		\parens{6.18}^{1/2}\parens{5.81}^{1/2}\parens{-0.0665} & 5.81+10.63\frac{\parens{X_t^{(1)}}^2}{1+\parens{X_t^{(1)}}^2} + 0.98\frac{\parens{X_t^{(2)}}^2}{1+\parens{X_t^{(2)}}^2}
		\end{matrix}}\dop w_t,
	\\
	\crotchet{\begin{matrix}
		X_0^{(1)}\\
		X_0^{(2)}
		\end{matrix}} = \crotchet{\begin{matrix}
		x_0^{(1)}\\
		x_0^{(2)}
		\end{matrix}}.
	\end{cases}
\end{align*}

Firstly, we check the existence of noise in observation. The $z$-value of test for noise detection in \citep{NU17} is $1244.375$ and it is so large value that we can reject the null hypothesis stating $\Lambda=O$ with ordinary significance level. Hence we are motivated to use local mean method for fitting rather than the local Gaussian approximation as \citep{K97}.

The results of hypothesis testing are identical to those in the previous section except for centricity. 
The random perturbation is dependent on the state $X$ and it leads to the motivation for SDE modelling. 
Moreover, both of the processes with respect to x-axis and y-axis are dependent to each other; 
therefore, it is necessary to model this phenomenon with a $2$-dimensional diffusion process. 
In comparison to the data in July 2017, this dataset is characterised with its non-centricity. We can observe the constant tendency of wind velocity throughout the observed term.
\begin{table}[h]
	\centering
	\begin{tabular}{c|r|l}
		test & \multicolumn{1}{c}{test statistic} & \multicolumn{1}{|c}{$p$-value}\\\hline
		OU & $\mathcal{T}_{1,n}=20154.56$ & $p<10^{-16}$\\
		centricity & $\mathcal{T}_{2,n}=28.9719$ & $p=5.11\times10^{-7}$\\
		independence & $\mathcal{T}_{1,n}+\mathcal{T}_{2,n}=3395.082$ & $p<10^{-16}$
	\end{tabular}
	\caption{Summary of the tests in Section 5}\label{table5}
\end{table}

\section{Conclusion}
We suggested some types of test statistics for parametric hypotheses in the use of some results for quasi-likelihood proposed in \citep{NU17}. 
In section for simulation and real data analysis, we examined the asymptotics of those statistics with practical hypotheses settings such that test for Ornstein-Uhlenbeck processes which can check the motivation to use diffusion modelling rather than classical time series modelling, and independence test which enables us to see whether we should model observed phenomena with multi-dimensional settings. 
In addition, centricity test which corresponds to the classical i.i.d. setting was used to see the process is centred along with zero vector or not.
With these tools for statistical analysis, we will obtain statistically-supported conclusion from high-frequency data even with the existence of observation noise.

\section{Proof}
In the following discussion, we denote 
\begin{align*}
I^{(2,2)}(\vartheta^{\star})&=\frac{9}{8}J^{(2,2)}(\vartheta^{\star})\\
I^{(3,3)}(\vartheta^{\star})&=J^{(3,3)}(\vartheta^{\star}).
\end{align*}

\begin{proof}[Proof of Theorem \ref{l0}]
	We only consider the asymptotics of $\mathbb{L}_{1,n}$ with $r_1>0$ since the case of $\mathbb{L}_{2,n}$ with $r_2>0$ is quite analogous. For Taylor's theorem, we obtain
	\begin{align*}
	\mathbb{L}_{1,n}\parens{\tilde{\alpha}_n|\hat{\Lambda}_n}&=\mathbb{L}_{1,n}\parens{\hat{\alpha}_n|\hat{\Lambda}_n}+\partial_{\alpha}\mathbb{L}_{1,n}\parens{\hat{\alpha}_n|\hat{\Lambda}_n}\parens{\tilde{\alpha}_n-\hat{\alpha}_n}\\
	&\qquad+\ip{\parens{\int_{0}^{1}\frac{(1-u)}{k_n}\partial_{\alpha}^2\mathbb{L}_{1,n}\parens{\hat{\alpha}_n+u\parens{\tilde{\alpha}_n-\hat{\alpha}_n}|\hat{\Lambda}_n}\dop u}}{\crotchet{\sqrt{k_n}\parens{\tilde{\alpha}_n-\hat{\alpha}_n}}^{\otimes 2}}
	\end{align*}
	We use the notation
	\begin{align*}
	\tilde{J}_{i,n}^{(2,2)}\parens{\tilde{\alpha}_n,\hat{\alpha}_n}:=-2\int_{0}^{1}\frac{(1-u)}{k_n}\partial_{\alpha}^2\mathbb{L}_{1,n}\parens{\hat{\alpha}_n+u\parens{\tilde{\alpha}_n-\hat{\alpha}_n}|\hat{\Lambda}_n}\dop u,
	\end{align*}
	and then under $H_0$, the consistency of $\hat{\alpha}_n$ and $\tilde{\alpha}_n$, and the discussion in the proof of Theorem 3.1.3 in \citep{NU17} lead to
	\begin{align*}
	\tilde{J}_{i,n}^{(2,2)}\parens{\tilde{\alpha}_n,\hat{\alpha}_n} \cp J^{(2,2)}(\vartheta^{\star}).
	\end{align*}
	We can evaluate $\partial_{\alpha}\mathbb{L}_{1,n}\parens{\hat{\alpha}_n|\hat{\Lambda}_n}=\mathbf{0}^T$ and hence
	\begin{align*}
	\mathcal{T}_{1,n}&=\frac{8}{9}\ip{\tilde{J}_{i,n}^{(2,2)}\parens{\tilde{\alpha}_n,\hat{\alpha}_n}}{\crotchet{\sqrt{k_n}\parens{\tilde{\alpha}_n-\hat{\alpha}_n}}^{\otimes 2}}.
	\end{align*}
	Then the result for the simplest case where $r_1=m_1$, i.e., $\alpha^{\star}=\mathbf{0}$, can be led since we have $\sqrt{k_n}\parens{\hat{\alpha}_n-\alpha^{\star}}\cl 
	\parens{J^{(2,2)}}^{-1} \parens{I^{(2,2)}}^{1/2}(\vartheta^{\star})Z$,
	where $Z\sim N(0,I_{m_1})$ because of Theorem 3.1.3 in \citep{NU17} and then
	\begin{align*}
	&\frac{8}{9}\ip{\tilde{J}_{i,n}^{(2,2)}\parens{\tilde{\alpha}_n,\hat{\alpha}_n}}{\crotchet{\sqrt{k_n}\parens{\tilde{\alpha}_n-\hat{\alpha}_n}}^{\otimes 2}}\\
	&\cl \frac{8}{9}\ip{\parens{J^{(2,2)}(\vartheta^{\star})} }{
		\crotchet{\parens{J^{(2,2)}}^{-1} \parens{I^{(2,2)}}^{1/2}(\vartheta^{\star})Z}^{\otimes 2}}\\
	&= \frac{8}{9}Z^T\crotchet{\parens{I^{(2,2)}}^{1/2}\parens{J^{(2,2)}}^{-1} \parens{I^{(2,2)}}^{1/2}(\vartheta^{\star})}Z\\
	&= \frac{8}{9}Z^T\crotchet{\parens{I^{(2,2)}}^{1/2}\parens{\frac{8}{9}I^{(2,2)}}^{-1} \parens{I^{(2,2)}}^{1/2}(\vartheta^{\star})}Z\\
	&\sim \chi_{m_1}^2.
	\end{align*}
	In general, it is necessary to examine the asymptotic behaviour of $\sqrt{k_n}\parens{\tilde{\alpha}_n-\hat{\alpha}_n}$. Let us consider the expansion
	\begin{align*}
	\frac{1}{\sqrt{k_n}}\parens{\partial_{\alpha}\mathbb{L}_{1,n}}^T\parens{\tilde{\alpha}_n|\hat{\Lambda}_n}&=\frac{1}{\sqrt{k_n}}\parens{\partial_{\alpha}\mathbb{L}_{1,n}}^T\parens{\hat{\alpha}_n|\hat{\Lambda}_n}\\
	&\qquad+\parens{\int_{0}^{1}\frac{1}{k_n}\partial_{\alpha}^2\mathbb{L}_{1,n}\parens{\hat{\alpha}_n+u\parens{\tilde{\alpha}_n-\hat{\alpha}_n}|\hat{\Lambda}_n}\dop u}\sqrt{k_n}\parens{\tilde{\alpha}_n-\hat{\alpha}_n}\\
	&=-\tilde{J}_{ii,n}^{(2,2)}\parens{\tilde{\alpha}_n,\hat{\alpha}_n}\sqrt{k_n}\parens{\tilde{\alpha}_n-\hat{\alpha}_n},
	\end{align*}
	where
	\begin{align*}
	\tilde{J}_{ii,n}^{(2,2)}\parens{\tilde{\alpha}_n,\hat{\alpha}_n}:=-\int_{0}^{1}\frac{1}{k_n}\partial_{\alpha}^2\mathbb{L}_{1,n}\parens{\hat{\alpha}_n+u\parens{\tilde{\alpha}_n-\hat{\alpha}_n}|\hat{\Lambda}_n}\dop u
	\end{align*}
	with the property
	\begin{align*}
	\tilde{J}_{ii,n}^{(2,2)}\parens{\tilde{\alpha}_n,\hat{\alpha}_n} \cp J^{(2,2)}(\vartheta^{\star}).
	\end{align*}
	Hence
	\begin{align*}
	-\parens{\tilde{J}_{ii,n}^{(2,2)}\parens{\tilde{\alpha}_n,\hat{\alpha}_n}}^{-1}\frac{1}{\sqrt{k_n}}\parens{\partial_{\alpha}\mathbb{L}_{1,n}}^T\parens{\tilde{\alpha}_n|\hat{\Lambda}_n}=\sqrt{k_n}\parens{\tilde{\alpha}_n-\hat{\alpha}_n}.
	\end{align*}
	It leads to
	\begin{align*}
	\mathcal{T}_{1,n}=\frac{8}{9}\ip{\tuborg{\crotchet{\parens{\tilde{J}_{ii,n}^{(2,2)}}^{-1}\parens{\tilde{J}_{i,n}^{(2,2)}}\parens{\tilde{J}_{ii,n}^{(2,2)}}^{-1}}\parens{\tilde{\alpha}_n,\hat{\alpha}_n}}}{\crotchet{\frac{1}{\sqrt{k_n}}\parens{\partial_{\alpha}\mathbb{L}_{1,n}}^T\parens{\tilde{\alpha}_n|\hat{\Lambda}_n}}^{\otimes2}}.
	\end{align*}
	Moreover, we check the expansion
	\begin{align*}
	\frac{1}{\sqrt{k_n}}\parens{\partial_{\alpha}\mathbb{L}_{1,n}}^T\parens{\tilde{\alpha}_n|\hat{\Lambda}_n}
	&=\frac{1}{\sqrt{k_n}}\parens{\partial_{\alpha}\mathbb{L}_{1,n}}^T\parens{\alpha^{\star}|\hat{\Lambda}_n}\\
	&\qquad+\parens{\int_{0}^{1}\frac{1}{k_n}\partial_{\alpha}^2\mathbb{L}_{1,n}\parens{\alpha^{\star}+u\parens{\tilde{\alpha}_n-\alpha^{\star}}|\hat{\Lambda}_n}\dop u}\sqrt{k_n}\parens{\tilde{\alpha}_n-\alpha^{\star}}.
	\end{align*}
	Let us partition $J^{(2,2)}$ into
	\begin{align*}
	J^{(2,2)}=
	\crotchet{
		\begin{matrix}
		G_{1}^{(2,2)} & G_{2}^{(2,2)}\\
		\parens{G_{2}^{(2,2)}}^T & G_{3}^{(2,2)}
		\end{matrix}
	},
	\end{align*}
	where $G_{1}^{(2,2)}\in \Re^{r_1}\otimes \Re^{r_1}$, $G_{2}^{(2,2)}\in \Re^{r_1}\otimes \Re^{m_1-r_1}$, and $G_{3}^{(2,2)}\in \Re^{m_1-r_1}\otimes \Re^{m_1-r_1}$ and define
	\begin{align*}
	H^{(2,2)}:=\crotchet{
		\begin{matrix}
		O & O \\
		O & \parens{G_{3}^{(2,2)}}^{-1}
		\end{matrix}
	}.
	\end{align*}
	Since the last $m_1-r_1$ components of $\partial_{\alpha}\mathbb{L}_{1,n}\parens{\tilde{\alpha}_n|\hat{\Lambda}_n}$ are equal to zero for sufficiently large $n$, we obtain $H^{(2,2)}\parens{\partial_{\alpha}\mathbb{L}_{1,n}}^T\parens{\tilde{\alpha}_n|\hat{\Lambda}_n}=\mathbf{0}$ and hence
	\begin{align*}
	\mathbf{0}&=H^{(2,2)}\frac{1}{\sqrt{k_n}}\parens{\partial_{\alpha}\mathbb{L}_{1,n}}^T\parens{\alpha^{\star}|\hat{\Lambda}_n}
	\\
	&\quad+H^{(2,2)}\parens{\int_{0}^{1}\frac{1}{k_n}\partial_{\alpha}^2\mathbb{L}_{1,n}\parens{\alpha^{\star}+u\parens{\tilde{\alpha}_n-\alpha^{\star}}|\hat{\Lambda}_n}\dop u}\sqrt{k_n}\parens{\tilde{\alpha}_n-\alpha^{\star}},
	\end{align*}
	and then
	\begin{align*}
	&J^{(2,2)}H^{(2,2)}\frac{1}{\sqrt{k_n}}\parens{\partial_{\alpha}\mathbb{L}_{1,n}}^T\parens{\alpha^{\star}|\hat{\Lambda}_n}\\
	&=
	-J^{(2,2)}H^{(2,2)}\parens{\int_{0}^{1}\frac{1}{k_n}\partial_{\alpha}^2\mathbb{L}_{1,n}\parens{\alpha^{\star}+u\parens{\tilde{\alpha}_n-\alpha^{\star}}|\hat{\Lambda}_n}\dop u}\sqrt{k_n}\parens{\tilde{\alpha}_n-\alpha^{\star}}.
	\end{align*}
	Note that the first $r_1$ components of $\tilde{\alpha}_n$ and $\alpha^{\star}$ are equal to zero, and it leads to
	\begin{align*}
	&J^{(2,2)}H^{(2,2)}\parens{\int_{0}^{1}\frac{1}{k_n}\partial_{\alpha}^2\mathbb{L}_{1,n}\parens{\alpha^{\star}+u\parens{\tilde{\alpha}_n-\alpha^{\star}}|\hat{\Lambda}_n}\dop u}\sqrt{k_n}\parens{\tilde{\alpha}_n-\alpha^{\star}}\\
	&=\parens{\int_{0}^{1}\frac{1}{k_n}\partial_{\alpha}^2\mathbb{L}_{1,n}\parens{\alpha^{\star}+u\parens{\tilde{\alpha}_n-\alpha^{\star}}|\hat{\Lambda}_n}\dop u}\sqrt{k_n}\parens{\tilde{\alpha}_n-\alpha^{\star}}.
	\end{align*}
	Then we have
	\begin{align*}
	\frac{1}{\sqrt{k_n}}\parens{\partial_{\alpha}\mathbb{L}_{1,n}}^T\parens{\tilde{\alpha}_n|\hat{\Lambda}_n}
	&=\parens{I-J^{(2,2)}H^{(2,2)}}\frac{1}{\sqrt{k_n}}\parens{\partial_{\alpha}\mathbb{L}_{1,n}}^T\parens{\alpha^{\star}|\hat{\Lambda}_n}.
	\end{align*}
	Theorem 7.5.1 in \citep{NU17} leads to
	\begin{align*}
	\frac{1}{\sqrt{k_n}}\parens{\partial_{\alpha}\mathbb{L}_{1,n}}^T\parens{\alpha^{\star}|\hat{\Lambda}_n}\cl \parens{I^{(2,2)}}^{1/2}Z,
	\end{align*}
	where $Z\sim N(0,I_{m_1})$, and hence
	\begin{align*}
	\frac{1}{\sqrt{k_n}}\parens{\partial_{\alpha}\mathbb{L}_{1,n}}^T\parens{\tilde{\alpha}_n|\hat{\Lambda}_n}\cl \parens{I-J^{(2,2)}H^{(2,2)}}\parens{I^{(2,2)}}^{1/2}Z.
	\end{align*}
	Because $H^{(2,2)}J^{(2,2)}H^{(2,2)}=H^{(2,2)}$, we obtain
	\begin{align*}
	\mathcal{T}_{1,n}&\cl \frac{8}{9}Z^T\parens{I^{(2,2)}}^{1/2}\parens{I-J^{(2,2)}H^{(2,2)}}^T\parens{J^{(2,2)}}^{-1}\parens{I-J^{(2,2)}H^{(2,2)}}\parens{I^{(2,2)}}^{1/2}Z\\
	&=\frac{8}{9}Z^T\parens{I^{(2,2)}}^{1/2}\crotchet{\parens{J^{(2,2)}}^{-1}-H^{(2,2)}}\parens{I^{(2,2)}}^{1/2}Z\\
	&=Z^T\parens{J^{(2,2)}}^{1/2}\crotchet{\parens{J^{(2,2)}}^{-1}-H^{(2,2)}}\parens{J^{(2,2)}}^{1/2}Z.
	\end{align*}
	Note that
	\begin{align*}
	&\parens{J^{(2,2)}}^{1/2}\crotchet{\parens{J^{(2,2)}}^{-1}-H^{(2,2)}}\parens{J^{(2,2)}}^{1/2}\parens{J^{(2,2)}}^{1/2}\crotchet{\parens{J^{(2,2)}}^{-1}-H^{(2,2)}}\parens{J^{(2,2)}}^{1/2}\\
	&=\parens{J^{(2,2)}}^{1/2}\crotchet{\parens{J^{(2,2)}}^{-1}-H^{(2,2)}}\parens{J^{(2,2)}}^{1/2}
	\end{align*}
	and $\tr\tuborg{I-\parens{J^{(2,2)}}H^{(2,2)}}=r_1$. Hence we obtain
	\begin{align*}
	\mathcal{T}_{1,n}\cl \chi_{r_1}^2
	\end{align*}
	(see \citep{Fe96}). If $r_{2}=0$, it completes the proof. Otherwise, with the identical discussion, we obtain
	\begin{align*}
		\mathcal{T}_{1,n}+\mathcal{T}_{2,n}
		&=\frac{8}{9}\ip{\parens{J^{(2,2)}\parens{\vartheta^{\star}}}^{-1}}{
			\crotchet{\frac{1}{\sqrt{k_n}}\parens{\partial_{\alpha}\mathbb{L}_{1,n}}^T\parens{\tilde{\alpha}_n|\hat{\Lambda}_n}}^{\otimes2}}\\
		&\quad + \ip{\parens{J^{(3,3)}\parens{\vartheta^{\star}}}^{-1}}{
			\crotchet{
				\frac{1}{\sqrt{T_n}}\parens{\partial_{\beta}\mathbb{L}_{2,n}}^T\parens{\tilde{\beta}_n|\hat{\Lambda}_n,\hat{\alpha}_n}}^{\otimes2}}+o_{P}\parens{1}\\
		&=\frac{8}{9}\ip{\parens{J^{(2,2)}\parens{\vartheta^{\star}}}^{-1}}{
			\crotchet{\parens{I-J^{(2,2)}H^{(2,2)}}\frac{1}{\sqrt{k_n}}\parens{\partial_{\alpha}\mathbb{L}_{1,n}}^T\parens{\alpha^{\star}|\hat{\Lambda}_n}}^{\otimes2}}\\
		&\quad + \ip{\parens{J^{(3,3)}\parens{\vartheta^{\star}}}^{-1}}{
			\crotchet{
				\parens{I-J^{(3,3)}H^{(3,3)}}\frac{1}{\sqrt{T_n}}\parens{\partial_{\beta}
					\mathbb{L}_{2,n}}^T\parens{\beta^{\star}|\hat{\Lambda}_n,\hat{\alpha}_{n}}}^{\otimes2}}
				+o_{P}\parens{1},
	\end{align*}
	where
	\begin{align*}
			J^{(3,3)}=
		\crotchet{
			\begin{matrix}
			G_{1}^{(3,3)} & G_{2}^{(3,3)}\\
			\parens{G_{2}^{(3,3)}}^T & G_{3}^{(3,3)}
			\end{matrix}
		},
	\end{align*}
	$G_{1}^{(3,3)}\in \Re^{r_2}\otimes \Re^{r_2}$, $G_{2}^{(3,3)}\in \Re^{r_2}\otimes \Re^{m_2-r_2}$, $G_{3}^{(3,3)}\in \Re^{m_2-r_2}\otimes \Re^{m_2-r_2}$, and 
	\begin{align*}
	H^{(3,3)}:=\crotchet{
		\begin{matrix}
		O & O \\
		O & \parens{G_{3}^{(3,3)}}^{-1}
		\end{matrix}
	}.
	\end{align*}
	Therefore, the convergence in law
	\begin{align*}
		\crotchet{\begin{matrix}
			\frac{1}{\sqrt{k_n}}
			\parens{\partial_{\alpha}\mathbb{L}_{1,n}}^T\parens{\alpha^{\star}|\hat{\Lambda}_n}\\
			\frac{1}{\sqrt{T_n}}\parens{\partial_{\beta}\mathbb{L}_{2,n}}^T\parens{\beta^{\star}|\hat{\Lambda}_n,\hat{\alpha}_n}
		\end{matrix}}\cl N\parens{\mathbf{0},\crotchet{\begin{matrix}
			I^{(2,2)} & O\\
			O & I^{(3,3)}
			\end{matrix}}\parens{\vartheta^{\star}}}
	\end{align*}
	(see \citep{NU17}) and continuous mapping theorem lead to
	\begin{align*}
	\mathcal{T}_{1,n}+\mathcal{T}_{2,n}\cl \chi_{r_1+r_2}^2.
	\end{align*}
	Hence we obtain the result.
\end{proof}

\begin{proof}[Proof of Theorem \ref{l1}]
	Firstly we consider the case where $r_1>0$ and $^\exists \ell_1\in\tuborg{1,\cdots,r_1}$ such that $\parens{\alpha^{\star}}^{(\ell_1)}\neq 0$. Note that
	\begin{align*}
	\sup_{\alpha\in\Theta_1}\abs{\frac{1}{k_n}\parens{\mathbb{L}_{1,n}\parens{\hat{\alpha}_n|\hat{\Lambda}_n}
			-\mathbb{L}_{1,n}\parens{\alpha|\hat{\Lambda}_n}}+\mathbb{Y}_1(\alpha)}\cp 0,
	\end{align*}
	because of \citep{NU17}, and (A6) leads to for all $n$,
	\begin{align*}
	\frac{9}{16k_n}\mathcal{T}_{1,n}+\abs{\frac{9}{16k_n}\mathcal{T}_{1,n}+\mathbb{Y}(\tilde{\alpha}_n)}
	\ge -\mathbb{Y}(\tilde{\alpha}_n)
	\ge \chi \norm{\tilde{\alpha}_n-\alpha^{\star}}^2
	\ge \chi \abs{\parens{\alpha^{\star}}^{(\ell_1)}}^{2} > 0.
	\end{align*}
	since $\mathbb{Y}(\alpha)\le 0$ for all $\alpha$.
	Therefore, for all $M>0$, 
	\begin{align*}
	P\parens{\mathcal{T}_{1,n}\le M} 
	&=P\parens{\frac{9}{16k_n}\mathcal{T}_{1,n}+\abs{\frac{9}{16k_n}\mathcal{T}_{1,n}+\mathbb{Y}(\tilde{\alpha}_n)}\le \frac{9M}{16k_n}+\abs{\frac{9}{16k_n}\mathcal{T}_{1,n}+\mathbb{Y}(\tilde{\alpha}_n)}}\\
	&\le   P\parens{\chi \abs{\parens{\alpha^{\star}}^{(\ell_1)}}^{2}
		\le \frac{9M}{16k_n}+\abs{\frac{9}{16k_n}\mathcal{T}_{1,n}+\mathbb{Y}(\tilde{\alpha}_n)}}\\
	&\le P\parens{\chi \abs{\parens{\alpha^{\star}}^{(\ell_1)}}^{2}-\frac{9M}{16k_n}
		\le \sup_{\alpha\in\Theta_1}\abs{\frac{1}{k_n}\parens{\mathbb{L}_{1,n}\parens{\hat{\alpha}_n|\hat{\Lambda}_n}
				-\mathbb{L}_{1,n}\parens{\alpha|\hat{\Lambda}_n}}+\mathbb{Y}_1(\alpha)}},
	\end{align*}
	and since for any $M$ there exists sufficiently large $n$ such that $\chi \abs{\parens{\alpha^{\star}}^{(\ell_1)}}^{2}-\frac{9M}{16k_n}\ge \frac{1}{2}\chi \abs{\parens{\alpha^{\star}}^{(\ell_1)}}^{2}$, we have
	\begin{align*}
		P\parens{\mathcal{T}_{1,n}\le M} &\le P\parens{\frac{1}{2}\chi \abs{\parens{\alpha^{\star}}^{(\ell_1)}}^{2}\le \sup_{\alpha\in\Theta_1}\abs{\frac{1}{k_n}\parens{\mathbb{L}_{1,n}\parens{\hat{\alpha}_n|\hat{\Lambda}_n}
					-\mathbb{L}_{1,n}\parens{\alpha|\hat{\Lambda}_n}}+\mathbb{Y}_1(\alpha)}} + o(1)\\
				&\to 0
	\end{align*}
	as $n\to\infty$ because of the uniform convergence in probability shown above.
	This result and the analogous discussion for $\mathcal{T}_{2,n}$ for the case where $r_2>0$ and $^\exists \ell_2\in\tuborg{1,\cdots,r_2}$ such that $\parens{\beta^{\star}}^{(\ell_2)}\neq 0$ complete the proof.
\end{proof}

\begin{proof}[Proof of Theorem \ref{r0}]
	We only proof for the convergence of $\mathcal{R}_{1,n}$ with $r_1>0$ and use the sets of notation same as Theorem 1. Then it holds
	\begin{align*}
	\mathcal{R}_{1,n}&=\frac{1}{\sqrt{k_n}}\partial_{\alpha}\mathbb{L}_{1,n}\parens{\alpha^{\star}|\hat{\Lambda}_n}\parens{I-J^{(2,2)}H}\parens{-\frac{9}{8k_n}\partial_{\alpha}^2\mathbb{L}_{1,n}\parens{\hat{\alpha}_n|\hat{\Lambda}_n}}^{-1}\parens{I-J^{(2,2)}H}\\
	&\qquad\times\frac{1}{\sqrt{k_n}}\parens{\partial_{\alpha}\mathbb{L}_{1,n}}^T\parens{\alpha^{\star}|\hat{\Lambda}_n}\\
	&\cl Z^T\parens{I^{(2,2)}}^{1/2}\parens{I-J^{(2,2)}H}
	\parens{\frac{9}{8}J^{(2,2)}}^{-1}\parens{I-J^{(2,2)}H}
	\parens{I^{(2,2)}}^{1/2}Z\\
	&\sim \chi_{r_1}^{2}.
	\end{align*}
\end{proof}

\begin{proof}[Proof of Theorem \ref{r1}]
	We proof for the case where $r_1>0$ and there exists $\ell_1\in\tuborg{1,\cdots,r_1}$ such that $\parens{\alpha^{\star}}^{(\ell_1)}\neq 0$. Note the Taylor's expansion
	\begin{align*}
	\frac{1}{k_n}\parens{\partial_{\alpha}\mathbb{L}_{1,n}}^T\parens{\tilde{\alpha}_n|\hat{\Lambda}_n}&=\frac{1}{k_n}\parens{\partial_{\alpha}\mathbb{L}_{1,n}}^T\parens{\hat{\alpha}_n|\hat{\Lambda}_n}\\
	&\qquad+\parens{\int_{0}^{1}\frac{1}{k_n}\partial_{\alpha}^2\mathbb{L}_{1,n}\parens{\hat{\alpha}_n+u\parens{\tilde{\alpha}_n-\hat{\alpha}_n}|\hat{\Lambda}_n}\dop u}\parens{\tilde{\alpha}_n-\hat{\alpha}_n}\\
	&=\parens{\int_{0}^{1}\frac{1}{k_n}\partial_{\alpha}^2\mathbb{L}_{1,n}\parens{\hat{\alpha}_n+u\parens{\tilde{\alpha}_n-\hat{\alpha}_n}|\hat{\Lambda}_n}\dop u}\parens{\tilde{\alpha}_n-\hat{\alpha}_n}\\
	&=\parens{\int_{0}^{1}J^{(2,2)}\parens{\hat{\alpha}_n+u\parens{\tilde{\alpha}_n-\hat{\alpha}_n}|\vartheta^{\star}}\dop u}\parens{\tilde{\alpha}_n-\hat{\alpha}_n}+o_P(1)
	\end{align*}
	for compactness of $\Theta_1$. Then 
	\begin{align*}
	\frac{1}{k_n}\mathcal{R}_{1,n}&=\parens{\frac{1}{k_n}\partial_{\alpha}\mathbb{L}_{1,n}\parens{\tilde{\alpha}_n|\hat{\Lambda}_n}}
	\parens{-\frac{9}{8k_n}\partial_{\alpha}^2\mathbb{L}_{1,n}\parens{\hat{\alpha}_n|\hat{\Lambda}_n}}^{-1}
	\parens{\frac{1}{k_n}\partial_{\alpha}\mathbb{L}_{1,n}\parens{\tilde{\alpha}_n|\hat{\Lambda}_n}}^T\\
	&= \parens{\tilde{\alpha}_n-\hat{\alpha}_n}^T\parens{\int_{0}^{1}J^{(2,2)}\parens{\hat{\alpha}_n+u\parens{\tilde{\alpha}_n-\hat{\alpha}_n}|\vartheta^{\star}}\dop u}\\
	&\qquad\quad\times \parens{-\frac{9}{8k_n}\partial_{\alpha}^2\mathbb{L}_{1,n}\parens{\hat{\alpha}_n|\hat{\Lambda}_n}}^{-1}
	\parens{\int_{0}^{1}J^{(2,2)}\parens{\hat{\alpha}_n+u\parens{\tilde{\alpha}_n-\hat{\alpha}_n}|\vartheta^{\star}}\dop u}\parens{\hat{\alpha}_n-\tilde{\alpha}_n}\\
	&\qquad+o_P(1)\\
	&=\parens{\tilde{\alpha}_n-\hat{\alpha}_n}^T\parens{\int_{0}^{1}J^{(2,2)}\parens{\hat{\alpha}_n+u\parens{\tilde{\alpha}_n-\hat{\alpha}_n}|\vartheta^{\star}}\dop u}\\
	&\qquad\quad\times \parens{\frac{9}{8}J^{(2,2)}\parens{\hat{\alpha}_n|\vartheta^{\star}}}^{-1}
	\parens{\int_{0}^{1}J^{(2,2)}\parens{\hat{\alpha}_n+u\parens{\tilde{\alpha}_n-\hat{\alpha}_n}|\vartheta^{\star}}\dop u}\parens{\hat{\alpha}_n-\tilde{\alpha}_n}\\
	&\qquad+o_P(1)
	\end{align*}
	for compactness of $\Theta_1$ again. When we set
	\begin{align*}
	R_{1,n}&=k_n\parens{\tilde{\alpha}_n-\hat{\alpha}_n}^T\parens{\int_{0}^{1}J^{(2,2)}\parens{\hat{\alpha}_n+u\parens{\tilde{\alpha}_n-\hat{\alpha}_n}|\vartheta^{\star}}\dop u}\\
	&\qquad\times \parens{\frac{9}{8}J^{(2,2)}\parens{\hat{\alpha}_n|\vartheta^{\star}}}^{-1}
	\parens{\int_{0}^{1}J^{(2,2)}\parens{\hat{\alpha}_n+u\parens{\tilde{\alpha}_n-\hat{\alpha}_n}|\vartheta^{\star}}\dop u}\parens{\hat{\alpha}_n-\tilde{\alpha}_n},
	\end{align*}
	it holds
	\begin{align*}
	\frac{1}{k_n}R_{1,n}&=\parens{\tilde{\alpha}_n-\hat{\alpha}_n}^T\parens{\int_{0}^{1}J^{(2,2)}\parens{\hat{\alpha}_n+u\parens{\tilde{\alpha}_n-\hat{\alpha}_n}|\vartheta^{\star}}\dop u}\\
	&\qquad\times \parens{\frac{9}{8}J^{(2,2)}\parens{\hat{\alpha}_n|\vartheta^{\star}}}^{-1}
	\parens{\int_{0}^{1}J^{(2,2)}\parens{\hat{\alpha}_n+u\parens{\tilde{\alpha}_n-\hat{\alpha}_n}|\vartheta^{\star}}\dop u}\parens{\hat{\alpha}_n-\tilde{\alpha}_n}\\
	&\ge \frac{8}{9}\lambda_{\max}\parens{J^{(2,2)}}^{-1}\parens{\tilde{\alpha}_n-\hat{\alpha}_n}^T\parens{\int_{0}^{1}J^{(2,2)}\parens{\hat{\alpha}_n+u\parens{\tilde{\alpha}_n-\hat{\alpha}_n}|\vartheta^{\star}}\dop u}^2\parens{\tilde{\alpha}_n-\hat{\alpha}_n}\\
	&\ge \frac{8}{9}\lambda_{\max}\parens{J^{(2,2)}}^{-1}
	\inf_{\alpha\in\Theta_1}\lambda_{\min}\parens{J^{(2,2)}\parens{\alpha|\vartheta^{\star}}}^2
	\norm{\tilde{\alpha}_n-\hat{\alpha}_n}^2\\
	&\ge \frac{8}{9}\lambda_{\max}\parens{J^{(2,2)}}^{-1}
	\inf_{\alpha\in\Theta_1}\lambda_{\min}\parens{J^{(2,2)}\parens{\alpha|\vartheta^{\star}}}^2\abs{\hat{\alpha}_n^{(\ell_1)}}^2\\
	&= \frac{8}{9}\lambda_{\max}\parens{J^{(2,2)}}^{-1}
	\inf_{\alpha\in\Theta_1}\lambda_{\min}\parens{J^{(2,2)}\parens{\alpha|\vartheta^{\star}}}^2
	\abs{\parens{\alpha_{\star}}^{(\ell_1)}}^2 + o_P(1)\\
	&= Q\abs{\parens{\alpha_{\star}}^{(\ell_1)}}^2 + o_P(1),
	\end{align*}
	where $Q:=\frac{8}{9}\lambda_{\max}\parens{J^{(2,2)}}^{-1}
	\inf_{\alpha\in\Theta_1}\lambda_{\min}\parens{J^{(2,2)}\parens{\alpha|\vartheta^{\star}}}^2$.
	Hence we obtain for all $M>0$, $\delta>0$,
	\begin{align*}
	P\parens{\mathcal{R}_{1,n}\le M}
	&=P\parens{\frac{1}{k_n}\mathcal{R}_{1,n}\le \frac{M}{k_n}}\\
	&=P\parens{\tuborg{\frac{1}{k_n}\mathcal{R}_{1,n}\le \frac{M}{k_n}}\cap
		\tuborg{\abs{\frac{1}{k_n}\mathcal{R}_{1,n}-\frac{1}{k_n}R_{1,n}}>\frac{\delta}{3}}}\\
	&\qquad+P\parens{\tuborg{\frac{1}{k_n}\mathcal{R}_{1,n}\le \frac{M}{k_n}}\cap
		\tuborg{\abs{\frac{1}{k_n}\mathcal{R}_{1,n}-\frac{1}{k_n}R_{1,n}}\le \frac{\delta}{3}}}\\
	&\le P\parens{\abs{\frac{1}{k_n}\mathcal{R}_{1,n}-\frac{1}{k_n}R_{1,n}}>\frac{\delta}{3}}+P\parens{\frac{1}{k_n}R_{1,n}\le \frac{M}{k_n}+\frac{\delta}{3}}\\
	&\le P\parens{Q\abs{\hat{\alpha}_n^{(\ell_1)}}^2
		\le \frac{M}{k_n}+\frac{\delta}{3}}+o(1)\\
	&\le P\parens{Q\abs{\parens{\alpha_{\star}}^{(\ell_1)}}^2
		\le \frac{M}{k_n}+\frac{2\delta}{3}}+ P\parens{\abs{\abs{\hat{\alpha}_n^{(\ell_1)}}^2-\abs{\parens{\alpha_{\star}}^{(\ell_1)}}^2}
		\ge \frac{\delta}{3Q}}+o(1)\\
	&\le P\parens{Q\abs{\parens{\alpha_{\star}}^{(\ell_1)}}^2
		\le \frac{M}{k_n}+\frac{2\delta}{3}}+o(1).
	\end{align*}
	We can choose $\delta$ to suffice $\delta< Q\abs{\parens{\alpha_{\star}}^{(\ell_1)}}^2$. For any $M>0$, there exists sufficiently large $n$ such that $M/k_n\le\delta/3$. Hence
	\begin{align*}
		P\parens{\mathcal{R}_{1,n}\le M}
		\le P\parens{Q\abs{\parens{\alpha_{\star}}^{(\ell_1)}}^2
			\le \delta}+o(1)=o(1),
	\end{align*}
	and then we obtain the result.
\end{proof}

\begin{proof}[Proof of Theorem \ref{w0}]
	We also proof for the convergence of $\mathcal{W}_{1,n}$ with $r_1>0$ and use the sets of notation same as Theorem 1. It holds
	\begin{align*}
	\mathcal{W}_{1,n}
	&=\frac{1}{\sqrt{k_n}}\partial_{\alpha}\mathbb{L}_{1,n}\parens{\tilde{\alpha}_n|\hat{\Lambda}_n}
	\parens{\tilde{J}_{ii,n}^{(2,2)}\parens{\tilde{\alpha}_n,\hat{\alpha}_n}}^{-1}
	\parens{-\frac{9}{8k_n}\partial_{\alpha}^2\mathbb{L}_{1,n}\parens{\hat{\alpha}_n|\hat{\Lambda}_n}}\\
	&\qquad\times\parens{\tilde{J}_{ii,n}^{(2,2)}\parens{\tilde{\alpha}_n,\hat{\alpha}_n}}^{-1}
	\frac{1}{\sqrt{k_n}}\parens{\partial_{\alpha}\mathbb{L}_{1,n}}^T\parens{\tilde{\alpha}_n|\hat{\Lambda}_n}\\
	&\cl Z^T\parens{I^{(2,2)}}^{1/2}\parens{I-J^{(2,2)}H}
	\parens{J^{(2,2)}}^{-1}\\
	&\qquad\times\parens{\frac{9}{8}J^{(2,2)}}\parens{J^{(2,2)}}^{-1}\parens{I-J^{(2,2)}H}
	\parens{I^{(2,2)}}^{1/2}Z\\
	&=Z^T\parens{I^{(2,2)}}^{1/2}\parens{I-J^{(2,2)}H}
	\parens{\frac{9}{8}J^{(2,2)}}^{-1}\parens{I-J^{(2,2)}H}
	\parens{I^{(2,2)}}^{1/2}Z\\
	&\sim \chi_{r_1}^{2}.
	\end{align*}
\end{proof}

\begin{proof}[Proof of Theorem \ref{w1}]
	We proof for the case where $r_1>0$ and there exists $\ell_1\in\tuborg{1,\cdots,r_1}$ such that $\parens{\alpha^{\star}}^{(\ell_1)}\neq 0$. For compactness of $\Theta_1$ and consistency of $\hat{\alpha}_n$
	\begin{align*}
	\frac{1}{k_n}\mathcal{W}_{1,n}
	&=\parens{\hat{\alpha}_n-\tilde{\alpha}_n}^T
	\parens{-\frac{9}{8k_n}\partial_{\alpha}^2\mathbb{L}_{1,n}\parens{\hat{\alpha}_n|\hat{\Lambda}_n}}
	\parens{\hat{\alpha}_n-\tilde{\alpha}_n}\\
	&=\parens{\alpha^{\star}-\tilde{\alpha}_n}^T
	\parens{I^{(2,2)}}
	\parens{\alpha^{\star}-\tilde{\alpha}_n} + o_P(1).
	\end{align*}
	and
	\begin{align*}
	A_n
	&:=\parens{\alpha^{\star}-\tilde{\alpha}_n}^T
	\parens{I^{(2,2)}}
	\parens{\alpha^{\star}-\tilde{\alpha}_n}\\
	&\ge  \abs{\parens{\alpha^{\star}}^{(\ell_1)}}^2 \inf_{\norm{\mathbf{x}}=1}\parens{\mathbf{x}^T\parens{I^{(2,2)}}\mathbf{x}}.
	\end{align*}
	Note the assumption (A6) and (A7); then
	\begin{align*}
	\mathbb{Y}_1(\alpha)=\norm{\alpha-\alpha^{\star}}^{-2}\parens{\alpha-\alpha^{\star}}^T
	J^{(2,2)}
	\parens{\alpha-\alpha^{\star}} + o\parens{1},
	\end{align*}
	when we consider $\alpha\to\alpha^{\star}$. Therefore, $I^{(2,2)}$ is positive definite, and for all $M>0$ and $\delta>0$,
	\begin{align*}
	P\parens{\mathcal{W}_{1,n}\le M}
	&=P\parens{\frac{1}{k_n}\mathcal{W}_{1,n}\le \frac{M}{k_n}}\\
	&=P\parens{\frac{1}{k_n}\mathcal{W}_{1,n} - A_n
		-\frac{M}{k_n} \le  -A_n}\\
	&\le P\parens{\tuborg{\frac{1}{k_n}\mathcal{W}_{1,n} - A_n
		-\frac{M}{k_n} \le  -A_n}\cap\tuborg{A_n> \delta}}
	+P\parens{A_n\le \delta}\\
	&=P\parens{\tuborg{\frac{1}{k_n}\mathcal{W}_{1,n} - A_n
			-\frac{M}{k_n} \le  -A_n}\cap\tuborg{-A_n< -\delta}}
	+P\parens{A_n\le \delta}\\
	&\le P\parens{\frac{1}{k_n}\mathcal{W}_{1,n} - A_n
		-\frac{M}{k_n} \le  -\delta}+
	P\parens{A_n\le \delta}\\
	&\le P\parens{\frac{1}{k_n}\mathcal{W}_{1,n} - A_n
		\le  \frac{M}{k_n} -\delta}+
	P\parens{A_n\le \delta}.
	\end{align*}
	For any $M>0$ and $\delta>0$, we have sufficiently large $n$ such that $M/k_n<\delta/2$, and then
	\begin{align*}
		P\parens{\frac{1}{k_n}\mathcal{W}_{1,n} - A_n
			\le  \frac{M}{k_n} -\delta}&\le P\parens{\frac{1}{k_n}\mathcal{W}_{1,n} - A_n
			\le  -\frac{\delta}{2}}+o(1)\\
		&= P\parens{ A_n-\frac{1}{k_n}\mathcal{W}_{1,n}
			\ge  \frac{\delta}{2}}+o(1)\\
		&\le P\parens{\abs{A_n-\frac{1}{k_n}\mathcal{W}_{1,n}}\ge  \frac{\delta}{2}}+o(1)\\
		&=o(1).
	\end{align*}
	Hence
	\begin{align*}
		P\parens{\mathcal{W}_{1,n}\le M}&\le P\parens{A_n\le \delta} + o(1)\\
		&\le P\parens{\abs{\parens{\alpha^{\star}}^{(\ell_1)}}^2 
			\inf_{\norm{\mathbf{x}}=1}\parens{\mathbf{x}^T\parens{I^{(2,2)}}\mathbf{x}}\le\delta }+o(1)
	\end{align*}
	Since $\delta$ is arbitrary, we obtain the proof.
\end{proof}

\section*{Acknowledgement}
This work 
was partially supported by 
JST CREST,
JSPS KAKENHI Grant Number 
JP17H01100 
and Cooperative Research Program
of the Institute of Statistical Mathematics.

\bibliography{bibliography}

\begin{thebibliography}{}

\bibitem[Bibby and S\o{}rensen, 1995]{BS95}
Bibby, B.~M. and S\o{}rensen, M. (1995).
\newblock Martingale estimating functions for discretely observed diffusion
  processes.
\newblock {\em Bernoulli}, 1:17--39.

\bibitem[Brockwell and Davis, 1991]{BD91}
Brockwell, P.~J. and Davis, R.~A. (1991).
\newblock {\em Time series: theory and methods}.
\newblock Springer, New York.

\bibitem[Eguchi and Masuda, 2018]{EM18}
Eguchi, S. and Masuda, H. (2018).
\newblock Schwarz type model comparison for {LAQ} models.
\newblock {\em Bernoulli}, 24(3):2278--2327.

\bibitem[Favetto, 2014]{Fa14}
Favetto, B. (2014).
\newblock Parameter estimation by contrast minimization for noisy observations
  of a diffusion process.
\newblock {\em Statistics}, 48(6):1344--1370.

\bibitem[Favetto, 2016]{Fa16}
Favetto, B. (2016).
\newblock Estimating functions for noisy observations of ergodic diffusions.
\newblock {\em Statistical Inference for Stochastic Processes}, 19:1--28.

\bibitem[Ferguson, 1996]{Fe96}
Ferguson, T.~S. (1996).
\newblock {\em A Course in Large Sample Theory}.
\newblock Chapman \& Hall, London.

\bibitem[Florens-Zmirou, 1989]{Fl89}
Florens-Zmirou, D. (1989).
\newblock Approximate discrete time schemes for statistics of diffusion
  processes.
\newblock {\em Statistics}, 20(4):547--557.

\bibitem[Fujii and Uchida, 2014]{FU14}
Fujii, T. and Uchida, M. (2014).
\newblock {AIC} type statistics for discretely observed ergodic diffusion
  processes.
\newblock {\em Statistical Inference for Stochastic Processes}, 17(3):267--282.

\bibitem[Gloter and Jacod, 2001a]{GlJ01a}
Gloter, A. and Jacod, J. (2001a).
\newblock Diffusions with measurement errors{. I.} local asymptotic normality.
\newblock {\em ESAIM{:} Probability and Statistics}, 5:225--242.

\bibitem[Gloter and Jacod, 2001b]{GlJ01b}
Gloter, A. and Jacod, J. (2001b).
\newblock Diffusions with measurement errors{. II.} optimal estimators.
\newblock {\em ESAIM{:} Probability and Statistics}, 5:243--260.

\bibitem[Jacod et~al., 2009]{JLMPV09}
Jacod, J., Li, Y., Mykland, P.~A., Podolskij, M., and Vetter, M. (2009).
\newblock Microstructure noise in the continuous case: the pre-averaging
  approach.
\newblock {\em Stochastic Processes and their Applications}, 119(7):2249--2276.

\bibitem[Kessler, 1995]{K95}
Kessler, M. (1995).
\newblock Estimation des parametres d'une diffusion par des contrastes
  corriges.
\newblock {\em Comptes rendus de l'Acad\'{e}mie des sciences. S\'{e}rie 1,
  Math\'{e}matique}, 320(3):359--362.

\bibitem[Kessler, 1997]{K97}
Kessler, M. (1997).
\newblock Estimation of an ergodic diffusion from discrete observations.
\newblock {\em Scandinavian Journal of Statistics}, 24:211--229.

\bibitem[Kitagawa and Uchida, 2014]{KU14}
Kitagawa, H. and Uchida, M. (2014).
\newblock Adaptive test statistics for ergodic diffusion processes sampled at
  discrete times.
\newblock {\em Journal of Statistical Planning and Inference}, 150:84--110.

\bibitem[Lehmann and Romano, 2005]{LR05}
Lehmann, E.~L. and Romano, J.~P. (2005).
\newblock {\em Testing Statistical Hypotheses}.
\newblock Springer Verlag, New York.

\bibitem[Nakakita and Uchida, 2017]{NU17}
Nakakita, S.~H. and Uchida, M. (2017).
\newblock Adaptive estimation and noise detection for an ergodic diffusion with
  observation noises. arxiv: 1711.04462.

\bibitem[{NWTC Information Portal}, 2018]{NWTC}
{NWTC Information Portal} (2018).
\newblock {NWTC} 135-m meteorological towers data repository.
  https://nwtc.nrel.gov/135mdata.

\bibitem[Uchida, 2010]{U10}
Uchida, M. (2010).
\newblock Contrast-based information criterion for ergodic diffusion processes
  from discrete observations.
\newblock {\em Annals of the Institute of Statistical Mathematics},
  62(1):161--187.

\bibitem[Uchida and Yoshida, 2012]{UY12}
Uchida, M. and Yoshida, N. (2012).
\newblock Adaptive estimation of an ergodic diffusion process based on sampled
  data.
\newblock {\em Stochastic Processes and their Applications}, 122(8):2885--2924.

\bibitem[Uchida and Yoshida, 2014]{UY14}
Uchida, M. and Yoshida, N. (2014).
\newblock Adaptive bayes type estimators of ergodic diffusion processes from
  discrete observations.
\newblock {\em Statistical Inference for Stochastic Processes}, 17(2):181--219.

\bibitem[Yoshida, 1992]{Y92}
Yoshida, N. (1992).
\newblock Estimation for diffusion processes from discrete observation.
\newblock {\em Journal of Multivariate Analysis}, 41(2):220--242.

\bibitem[Yoshida, 2011]{Y11}
Yoshida, N. (2011).
\newblock Polynomial type large deviation inequalities and quasi-likelihood
  analysis for stochastic differential equations.
\newblock {\em Annals of the Institute of Statistical Mathematics},
  63:431--479.

\end{thebibliography}

\end{document}